\documentclass{amsart}
\usepackage{amssymb,
enumitem,
mathrsfs,
mathtools,
tikz,
upgreek,
verbatim,
hyphenat,
}
\usepackage[shadow]{todonotes}
\usepackage[colorlinks,linkcolor=blue,citecolor=blue]{hyperref}
\usepackage[T1]{fontenc}
\usepackage[shadow]{todonotes}

\newcommand{\ZZ}{\mathbb{Z}}

\newcommand{\RR}{\mathbb{R}}

\newcommand{\HH}{\mathbb{H}}

\newcommand{\id}{\operatorname{id}}

\newcommand{\Mod}{\operatorname{Mod}}

\newcommand{\lra}{\longrightarrow}
\newcommand{\eand}{\quad \text{ and } \quad}

\newcommand{\Stab}{\operatorname{Stab}}

\newenvironment{enumerate-(a)}{\begin{enumerate}[label={\upshape (\alph*)}, leftmargin=2pc]}{\end{enumerate}}
\newenvironment{enumerate-(a)-r}{\begin{enumerate}[label={\upshape (\alph*)}, leftmargin=2pc,resume]}{\end{enumerate}}
\newenvironment{enumerate-(a)-5}{\begin{enumerate}[label={\upshape (\alph*)}, leftmargin=2pc,start=5]}{\end{enumerate}}
\newenvironment{enumerate-(A)}{\begin{enumerate}[label={\upshape (\Alph*)}, leftmargin=2pc]}{\end{enumerate}}
\newenvironment{enumerate-(A)-r}{\begin{enumerate}[label={\upshape (\Alph*)}, leftmargin=2pc,resume]}{\end{enumerate}}
\newenvironment{enumerate-(i)}{\begin{enumerate}[label={\upshape (\roman*)}, leftmargin=2pc]}{\end{enumerate}}
\newenvironment{enumerate-(i)-r}{\begin{enumerate}[label={\upshape (\roman*)}, leftmargin=2pc,resume]}{\end{enumerate}}
\newenvironment{enumerate-(I)}{\begin{enumerate}[label={\upshape (\Roman*)}, leftmargin=2pc]}{\end{enumerate}}
\newenvironment{enumerate-(I)-r}{\begin{enumerate}[label={\upshape (\Roman*)}, leftmargin=2pc,resume]}{\end{enumerate}}
\newenvironment{enumerate-(1)}{\begin{enumerate}[label={\upshape (\arabic*)}, leftmargin=2pc]}{\end{enumerate}}
\newenvironment{enumerate-(1)-r}{\begin{enumerate}[label={\upshape (\arabic*)}, leftmargin=2pc,resume]}{\end{enumerate}}

\newtheorem{theorem}{Theorem}[section]
\newtheorem{lemma}[theorem]{Lemma}
\newtheorem{corollary}[theorem]{Corollary}
\newtheorem{proposition}[theorem]{Proposition}
\newtheorem{question}[theorem]{Question}

\theoremstyle{definition}
\newtheorem{definition}[theorem]{Definition}

\theoremstyle{remark}
\newtheorem{remark}[theorem]{Remark}

\begin{document}

\title[]{Rigidity of braid group actions on $\mathbb{R}$ and of low-genus mapping class group actions on $S^1$}

\date{}

\author[I.~Ba]{Idrissa Ba}

\email{ba162006@yahoo.fr}

\author[A.~Clay]{Adam Clay}

\address{Department of Mathematics, 420 Machray Hall, University of
Manitoba, Winnipeg, MB, R3T 2N2}
\email{Adam.Clay@umanitoba.ca}

\author[T.~Ghaswala]{Tyrone Ghaswala}
\address{Center for Education in Mathematics and Computing,
University of Waterloo,
Waterloo,
ON Canada N2L 3G1}
\email{tghaswala@uwaterloo.ca}

\thanks{Adam Clay was partially supported by NSERC grant RGPIN-05343-2020.}

\begin{abstract} 
Every nontrivial action of the braid group $B_n$ on $\mathbb{R}$ by orientation-preserving homeomorphisms yields, up to conjugation by a homeomorphism of $\mathbb{R}$, a representation $\rho : B_n \rightarrow \mathrm{H\widetilde{ome}o}_+(S^1)$ and therefore determines a translation number for every element of $B_n$.  In this manuscript we offer a simple characterisation of which actions of $B_n$ on $\mathbb{R}$ produce translation numbers that agree with those arising from the standard Nielsen-Thurston action on $\mathbb{R}$.  Our approach is to prove an analogous statement concerning left orderings of $B_n$ via a technique that uses the space of left orderings of $B_n$, the isolated points in this space, and the natural conjugacy action of $B_n$.  We use this result to extend recent rigidity results of Mann and Wolff concerning mapping class group actions on $S^1$ to the case of low-genus surfaces with marked points.
\end{abstract}

\maketitle

\section{Introduction}

We use $\Sigma_{g,n}^b$ to denote a surface of genus $g$ with $n$ marked points and $b$ boundary components, and $\Mod(\Sigma_{g,n}^b)$ its mapping class group.  It is well known that if $b>0$, then $\Mod(\Sigma_{g,n}^b)$ is a left-orderable group \cite{SW}.  The second and third authors showed in \cite{CG} that when $b=1$, the mapping class $T_d$ of a Dehn twist about a curve $d$ parallel to the boundary component is cofinal in every left ordering $<$ of $\Mod(\Sigma_{g,n}^b)$, meaning that 
\[ \Mod(\Sigma_{g,n}^b) = \{ \gamma \mid \exists k \in \mathbb{Z} \mbox{ such that } T_d^{-k} < \gamma < T_d^k \}.
\]

This fact allows for a type of correspondence when $g>0$ between left orderings of $\Mod(\Sigma_{g,0}^1)$ and circular orderings of $\Mod(\Sigma_{g,1}^0)$ by taking the quotient by the central subgroup $\langle T_d \rangle$.  Allowing marked points in the $g=0$ case establishes a similar correspondence between left orderings of the braid group $B_n \cong \Mod(\Sigma_{0,n}^1)$ where $n \geq 2$ and circular orderings of $B_n / \langle \Delta_n^2 \rangle$, which corresponds to the subgroup of $\Mod(\Sigma_{0,n+1}^0)$ whose elements fix a chosen marked point (here, $\Delta_n^2$ is the square of the Garside element and is a generator of the center of $B_n$ when $n \geq 3$).

Put another way, there is a correspondence between orientation-preserving actions of $\Mod(\Sigma_{g,0}^1)$ on $\mathbb{R}$ and orientation-preserving actions of $\Mod(\Sigma_{g,1}^0)$ on $S^1$, via the ``capping homomorphism'' when $g>0$. A similar correspondence occurs when $g=0$ upon restriction to an appropriate subgroup of $\Mod(\Sigma_{0,n+1}^0)$. 
This correspondence was key in \cite{CG}, which showed that (up to sign) the translation number of an element in $\Mod(\Sigma_{g,0}^1)$ is independent of the underlying action on $\mathbb{R}$ when $g \geq 2$.  This allows one to conclude, for instance, that one of the standard definitions of the fractional Dehn twist coefficient is independent of the underlying action (See Section \ref{FDTC} for this definition).

The genus restriction in that result arises naturally in the course of the proof, from an application of Mann and Wolff's \cite{MW} rigidity result concerning actions of $\Mod(\Sigma_{g,1}^0)$ on $S^1$.   In their work, they completely characterise which actions of $\Mod(\Sigma_{g,1}^0)$ on $S^1$ are semiconjugate to the standard action when $g \geq 2$.
 The essential step in \cite{CG} is  then translating their result into a corresponding fact about actions of $\Mod(\Sigma_{g,0}^1)$ on $\mathbb{R}$ where $g \geq 2$.  In the low-genus cases, that is when $g =0$ or $g=1$, it is easy to see from examples that the translation number of an element in $\Mod(\Sigma_{1,0}^1)$ or $\Mod(\Sigma_{0,n}^1)$ in \emph{not} independent of the underlying action \cite[Section 6]{CG}.  

Nonetheless, in the present paper we are able to complete the picture with an analysis of the genus $0$ and $1$ cases.  Whereas the technique described above involved applying an established rigidity result concerning actions on $S^1$ in order to arrive at an analogous statement about actions of a certain central extension on $\mathbb{R}$, in this manuscript our arguments work in the reverse manner.  Specifically, we introduce a new technique for establishing a type of rigidity of circle actions that begins with an analysis of the central extension and its space of left orderings.  To be more explicit, we introduce some notation.

Denote the space of left orderings of $B_n$ by $\mathrm{LO}(B_n)$, which we think of as the collection of all positive cones $P \subset B_n$.  Owing to the fact that the central element $\Delta_n^2$ is cofinal in every left ordering of $B_n$, for every $\beta \in B_n$ we can define the translation number $\tau_{\beta} : \mathrm{LO}(B_n) \rightarrow \mathbb{R}$, which turns out to be a continuous map (Proposition \ref{continuity}).  
Let $P_D \in \mathrm{LO}(B_n)$ denote the positive cone of the Dehornoy ordering. By combining continuity of $\tau_{\beta}$ with properties of the positive cone $P_D$ and its associated isolated point $P_{DD}$, the positive cone of the Dubrovina-Dubrovin ordering, we prove the following.  In Theorem \ref{intro:LO-rigidity} below, we implicitly identify $B_{n-1}$ with a subgroup of $B_n$ in the usual way, so that we can think of the generator $\Delta_{n-1}^2$ 
 of the centre of $B_{n-1}$ as an element of $B_n$.

\begin{theorem}\label{intro:LO-rigidity}
Suppose that $P \in \mathrm{LO}(B_n)$ and $\Delta_n^2 \in P$,  and that $\Delta_{n-1}^2$ is not cofinal in the left ordering of $B_n$ determined by $P$.  Then $\tau_{\beta}(P) = \tau_{\beta}(P_D)$ for all $\beta \in B_n$.
\end{theorem}

As an immediate consequence, we can determine exactly which left orderings of $B_n = \mathrm{Mod}(\Sigma_{0,n}^1)$ can be used to compute the fractional Dehn twist coefficient of an element $h \in \mathrm{Mod}(\Sigma_{0,n}^1)$ in terms of the translation number of $h$ determined by the left ordering.  We can state our result as follows.  We use $c(h)$ to denote the fractional Dehn twist coefficient of $h \in \Mod(\Sigma_{g,n}^1)$ and $T_{\alpha}$ the mapping class of a Dehn twist about a simple closed curve $\alpha$ in $\Sigma_{g,n}^b$.

\begin{theorem}
\label{thm:mappingclassorders}
Suppose that $P \subset  \mathrm{Mod}(\Sigma_{0,n}^1)$ is the positive cone of a left ordering $<$ and that $T_d >id$.  If there exists a simple closed curve $\alpha \subset  \Sigma_{0,n}^1 $ such that: 
\begin{enumerate}
\item  $\Sigma_{0,n}^1 \setminus \{ \alpha \} \cong \Sigma_{0,n-1}^1 \cup \Sigma_{0,1}^2$, and
\item $T_{\alpha}$ is not cofinal in the ordering $<$ of $  \mathrm{Mod}(\Sigma_{0,n}^1)$,
\end{enumerate}
then $c(h) = \tau_h(P)$ for all $h \in  \mathrm{Mod}(\Sigma_{0,n}^1)$.  
\end{theorem}
  This builds upon the results of \cite{Mal}, which shows that the Dehornoy ordering can be used to compute fractional Dehn twist coefficients of elements of $\mathrm{Mod}(\Sigma_{0,n}^1)$, and provides a uniqueness result for pseudocharacters $B_n \rightarrow \mathbb{R}$ provided the character takes on nonnegative values for all braids in $P_D$ \cite[Theorem 8.1]{Mal}.  This also contrasts with the results of \cite{CG}, which show that for surfaces of genus $g \geq 2$ any left ordering will serve to compute fractional Dehn twist coefficients.  There is an analogous statement in the case of $\mathrm{Mod}(\Sigma_{1,0}^1)$, owing to the fact that this group is isomorphic to $B_3$, see Theorem \ref{genus1}. 

We can also pass to the quotient $\Mod(\Sigma_{0,n+1}^0)$ using the capping homomorphism, and recover a rigidity result akin to that of Mann and Wolff \cite{MW}.  They determined exactly which orientation-preserving actions of $\Mod(\Sigma_{g,1}^0)$ on $S^1$ are semiconjugate to the standard action when $g \geq 2$, up to reversing orientation of the circle.  We answer that same question when $g=0$ or $g= 1$, by considering $\Mod(\Sigma_{1,1}^0)$, and the subgroup $H \subset \Mod(\Sigma_{0,n+1}^0)$ of mapping classes that fix a prescribed marked point (i.e., $H$ is the image of $\mathrm{Mod}(\Sigma_{0,n}^1)$ under the capping homomorphism).  

\begin{theorem}
\label{thm:intro_rigidity}
Let $H \subset \Mod(\Sigma_{0,n+1}^0)$ denote the image of $\mathrm{Mod}(\Sigma_{0,n}^1)$ under the capping homomorphism, and $*$ the special marked point arising from the capping homomorphism.  Let $\alpha$ denote a simple closed curve on $\Sigma_{0,n+1}$ that surrounds $n-1$ marked points, none of which are $*$.  If the rotation number of the Dehn twist about $\alpha$ is zero with respect to both $f_1$ and $f_2$,  then $[f_1] = \pm[f_2] \in H^2_b(H;\ZZ)$.
\end{theorem}

Stated more informally, this says that if rotation number of some curve surrounding $n-1$ marked points is zero, then this completely determines the circular ordering (equivalently, its corresponding action on $S^1$) up to semiconjugacy and reversing orientation.  Again, we can also resolve the genus-$1$ case by using the isomorphism between $\mathrm{Mod}(\Sigma_{1,0}^1)$ and $B_3$, see Corollary \ref{cor:genus-1}.

Note also that the rotation number assumption in Theorem 1.3 is necessary, in the sense that there exist circular orderings (equivalently, actions on $S^1$) for which curves encircling $n-1$ marked points do not have rotation number zero, and consequently these circular orderings/actions are not semiconjugate to the standard one.  Such examples can be realised easily from the examples of left orderings given in \cite[Section 6]{CG}.

Results of this flavour are not entirely surprising, as there is a known connection between isolated orderings and rigidity of their associated dynamical realisations \cite[Theorems 1.2 and 3.11]{MR}.  To our knowledge, however, this manuscript is the first to explicitly compute and exploit such a connection.  Our work also raises the following natural question.

\begin{question}
    When $b>0$, does $\mathrm{LO}(\Mod(\Sigma_{g,n}^b))$ admit isolated points?  If yes, can one recover the rigidity results of Mann and Wolff using a technique that exploits the isolated points in $\mathrm{LO}(\Mod(\Sigma_{g,0}^1))$?
\end{question}

\subsection{Organisation of the paper.} Section \ref{Background} contains background on orderable groups, including spaces of orderings, dynamic realisations and translation numbers.  In Section \ref{sec3} we show that the translation number is a continuous map on the space of orderings.  Section \ref{sec:dehornoy} introduces the Dehornoy ordering, the Dubrovina-Dubrovin ordering, and contains the proof of Theorem \ref{intro:LO-rigidity}.  In Section \ref{FDTC} we apply this result to fractional Dehn twist coefficients, and in Section \ref{sec:MCG} we prove a rigidity result for actions of low-genus mapping class groups on $S^1$.






\section{Background}
\label{Background}

\subsection{Left-orderable groups and spaces of orderings}

Recall that a group $G$ is \emph{left orderable} if there exists a strict total ordering $<$ of the elements of $G$ such that $g<h$ implies $fg<fh$ for all $f, g, h \in G$.  Equivalently, $G$ is left orderable if there exists a \emph{positive cone} $P \subset G$, satisfying $P \cdot P \subset P$, $G \setminus \{id \} = P \cup P^{-1}$ and $P \cap P^{-1} = \emptyset$.  The equivalence of these two definitions follows from observing that if $<$ is a given ordering of a group $G$, then $P_< = \{ g \in G \mid g>1\}$ is a positive cone; conversely if we are given a positive cone $P$ then $g<h \iff g^{-1}h \in P$ defines a strict total ordering of $G$ that is left-invariant.

A subgroup $C$ of a group $G$ is \emph{$<$-convex} (or convex with respect to $<$)  if for all $c,d \in C$ and $g \in G$, $c<g<d$ implies $g \in C$.  A subgroup $C$ of $G$ is \emph{relatively convex} if there exists an ordering $<$ of $G$ such that $C$ is $<$-convex.  We introduce the following definition for ease of discussion later in the manuscript.

\begin{definition}
\label{relative def}
Let $G$ be a group and $C$ a proper subgroup.  A \emph{relative left ordering of $G$ with respect to $C$} is a strict total ordering $<$ of the set of left cosets $G/C$ such that $gC <hC$ implies $fgC<fhC$ for all $f, g, h \in G$.  When such an ordering exists, we say that $G$ is \emph{relatively left orderable with respect to $C$}.
\end{definition}
 
It is a straightforward exercise to verify that if $G$ is left orderable, then $G$ admits a relative left ordering with respect to $C$ if and only if $C$ is a left-relatively convex subgroup.
  
Equivalently, $G$ is relatively left orderable with respect to a proper subgroup $C$ if there exists a \emph{positive cone relative to $C$}, that is, a subset $P \subset G$ satisfying:
\begin{enumerate}[label=(\roman*)]
\item $P \neq \{ id \}$,
\item $CPC \subset P$,
\item $G = C \sqcup P\sqcup P^{-1}$, where $\sqcup$ indicates a disjoint union. 
\end{enumerate}

There is a correspondence between relative left orderings as in Definition \ref{relative def} and nonempty subsets $P \subset G$ satisfying (i)--(iii) above.   The correspondence appears in \cite[Lemma 2.1]{AR21}.




We define the space of \emph{relative left orderings}, denoted $\mathrm{LO}_{rel}(G)$, as follows. Set
\[\mathrm{LO}_{rel}(G) = \{ P \subset G \mid \mbox{ $\exists$  $C \leq G$, $C \neq G$, s.t. $P$ is a positive cone relative to $C$} \},
\]
note that this is a subset of the power set $\mathcal{P}(G)$.  We equip the power set $\mathcal{P}(G) = \{0,1\}^G$ with the product topology arising from the discrete topology on $\{0,1\}$, and $\mathrm{LO}_{rel}(G)$ with the subspace topology.  The subbasic open sets are therefore:
\[ U_g = \{ P \in \mathrm{LO}_{rel}(G) \mid g \in P\}, \mbox{ and } U_{g^{-1}} = \{ P \in \mathrm{LO}_{rel}(G) \mid g^{-1} \in P\}.
\]

Equipped with the subspace topology, $\mathrm{LO}_{rel}(G)$ is not necessarily compact, however  $\mathrm{LO}_{rel}(G)$ is compact as long as $G$ is finitely generated  \cite[Theorem 1.4]{AR21}.   We similarly define the space of left orderings $$\mathrm{LO}(G) = \{P \subset G \mid \mbox{$P$ is a positive cone} \}$$ and topologise it again using the subspace topology inherited from $\mathcal{P}(G) = \{0,1\}^G$. Note that there is a natural embedding $i: \mathrm{LO}(G) \rightarrow \mathrm{LO}_{rel}(G)$, since every positive cone $P \subset G$ is a positive cone relative to the subgroup $\{id \}$.  Moreover, each of the spaces $\mathrm{LO}_{rel}(G)$ and $\mathrm{LO}(G)$ comes equipped with a natural $G$-action by conjugation, $g \cdot P = gPg^{-1}$ for all $g \in G$, which one can check is an action by homeomorphisms.

\subsection{Dynamic realisations and translation numbers}

Given a left-ordered group $(G, <)$, a \emph{gap} in $G$ is a pair of elements $g, h \in G$ with $g<h$ such that there is no $f \in G$ with $g<f<h$.  We will call an order-preserving embedding $t:(G, <) \rightarrow (\mathbb{R}, <)$  \emph{tight} if for every interval $(a, b) \subset \mathbb{R} \setminus t(G)$ there exists a gap $g, h \in G$ with $(a, b) \subset (t(g), t(h))$.
Whenever $G$ is countable there exists a tight embedding $t:(G, <) \rightarrow (\mathbb{R}, <)$ via the usual `midpoint construction', see e.g. \cite[Chapter 2]{CR16}.

If $G$ is a countable left-ordered group, a \emph{dynamic realisation} of an ordering $<$ with positive cone $P$ is a homomorphism $\rho_P: G \rightarrow \mathrm{Homeo}_+(\mathbb{R})$ defined as follows. First, fix a choice of tight embedding $t$, and then define $\rho_P(g)(t(h)) = t(gh)$ for all $g, h \in G$, extend this action to $\overline{t(G)}$ by taking limits, and then affinely to $\mathbb{R} \setminus \overline{t(G)}$.  One can check that all this works out, and that any two dynamic realisations of the same order (arising from different choices of tight embedding) are conjugate to one another via a orientation-preserving homeomorphism $\mathbb{R} \rightarrow \mathbb{R}$.

Given $\rho: G \rightarrow  \mathrm{H\widetilde{ome}o}_+(S^1)$, we can define the translation number of $g \in G$ determined by $\rho$ as the limit 
\[ \tau_g(\rho) = \lim_{n \to \infty} \frac{\rho(g)(0)}{n}.
\]
When $\rho = \rho_P$ for some ordering with positive cone $P$, we can (provided the group $G$ satisfies a certain algebraic condition) define this same quantity algebraically.

Recall that an element $g$ in a left-ordered group $(G, <)$ is cofinal relative to the ordering $<$ if 
\[ G = \{ h \in G \mid \exists k \in \mathbb{Z} \mbox{ such that } g^{-k} < h < g^k \}.
\]
When $z \in G$ is central, positive and cofinal relative to the ordering $<$ of $G$ with positive cone $P$, we can define the \emph{floor} of $g \in G$ with respect to $<$ to be the unique integer $[g]_P$ such that 
\[ z^{[g]_P} \leq g < z^{[g]_P+1}.
\]
Then we set 
\[ \tau_g(P) = \lim_{n \to \infty}\frac{[g^n]_P}{n}.
\]
While these quantities depend on the choice of $z \in G$, we suppress $z$ from our notation as the positive cofinal element under consideration will be fixed in our applications that follow.

 It is straightforward to work out the following proposition from the definitions; we use $sh(1)$ to denote the function $sh(1)(x) = x+1$ for all $x \in \mathbb{R}$.

\begin{proposition}
\label{translation number def}
Suppose that $<$ is a left ordering of $G$ with positive cone $P$ and that $z \in G$ is positive, central and $<$-cofinal.  If $\rho_P$ is a dynamic realisation of $<$ with $\rho_P(z) = sh(1)$, then $\tau_g(P) = \tau_g(\rho_P)$.
\end{proposition}

We can therefore tackle questions about translation numbers algebraically, using left orderings of $G$ and cofinal, central elements.  We observe that conjugation invariance of the translation number (a classical result) implies the following in our algebraic setting. 

\begin{proposition}
\label{conjugation}
Suppose that $<$ is a left ordering of $G$ with positive cone $P$ and that $z \in G$ is positive, central and $<$-cofinal. 
For all  $g,h \in G$, we have $\tau_{h}(P) = \tau_{h}(g P g^{-1})$.
\end{proposition}
\begin{proof} Let $P \in \mathrm{LO}(G)$ and $g, h \in G$.  For ease of notation, set $g Pg^{-1} = Q$. For any $k\in \mathbb{N}$ we have 
\[ z^{[h^k]_{P}} \leq_P h^k <_P z^{[h^k]_{P} + 1}
\]
which, by definition of $Q$, yields
\[ z^{[h^k]_{P}} \leq_Q g h^kg^{-1} <_Q z^{[h^k]_{P} + 1}
\]
and therefore $[(g hg^{-1})^k]_Q = [h^k]_P$.
It follows that $\tau_{h}(Q) = \tau_{g h g^{-1}}(Q) = \tau_{h}(P)$, where the equality $\tau_{h}(Q) = \tau_{g h g^{-1}}(Q)$ is precisely the classical fact that translation number is invariant under conjugation.
\end{proof}


\section{Continuity of translation number and its behaviour with respect to relative orderings}
\label{sec3}

For a left-orderable group $G$, we say an element $z \in G$ is \emph{absolutely cofinal} if $z$ is cofinal with respect to every left ordering of $G$.  Our interest in such elements stems from the fact that $\Delta_n^2 \in B_n$, which is a generator of the center of $B_n$ for $n \geq 3$, is absolutely cofinal \cite[Chapter II, Proposition 3.6]{DDRW}.

Let $G$ be a left-orderable group with a central, absolutely cofinal element $z$. Observe that $LO(G) = U_z \cup U_{z^{-1}}$ is a disjoint union, where $U_z$ contains precisely the orderings relative to which $z$ is positive. In this section we study the map $\tau_g:U_z \to \RR$, defined by the translation number for each $g \in G$, and show that it is continuous.  We first require a lemma.

\begin{lemma}
\label{continuity lemma}
Let $G$ be a countable left-orderable group with a central, absolutely cofinal element $z$. Let $P,Q \in U_z \subset LO(G)$ and let $m > 0$ be an arbitrary positive integer. If $g \in G$ satisfies $[g^k]_P = [g^k]_Q$ for all $1 \leq k \leq m$, then for all $\ell > m$ we have $\left|[g^\ell]_P - [g^\ell]_Q\right| < \frac{\ell}{m} + 1$. 
\end{lemma}

\begin{proof}
Write $\ell = sm+t$, where $s>0$ and $0\leq t < m$.  Set $r = [g^m]_P = [g^m]_Q$, and $v = [g^t]_P = [g^t]_Q$.  From the inequalities
\[ z^r \leq_P g^m <_P z^{r+1} \mbox{ and } z^v \leq_P g^t <_P z^{v+1}
\]
we compute
\[ z^{rs + v} \leq_P g^{sm+t} <_P z^{rs+s+v+1}.
\]
It follows that $rs+v \leq [g^{\ell}]_P \leq rs+s+v$.  A similar computation yields $rs+v \leq [g^{\ell}]_Q \leq rs+s+v$, so that $| [g^\ell]_P - [g^\ell]_Q| \leq s < \frac{\ell}{m}+1$.
\end{proof}

\begin{proposition}
\label{continuity}
Let $G$ be a countable left-orderable group with absolutely cofinal central element $z$. Given $g \in G$, the map $\tau_g:U_z \to \mathbb{R}$ is continuous.
\end{proposition}


\begin{proof}
Suppose that $\{P_q\}_{q=1}^{\infty}$ is a sequence of positive cones converging to $P$ in $U_{z}$, and fix $g \in G$.  We first consider the double limit $$\lim_{(q,k) \to \infty} \frac{[g^k]_{P_q}-[g^k]_{P}}{k}$$
and show that this limit tends to zero.  To this end, let $\epsilon >0$ and choose $m$ such that $\frac{1}{m} < \frac{\epsilon}{2}$.  Observe that we can choose $N_m$ so that $[g^k]_{P_q} = [g^k]_P$ for all $q>N_m$ and $1 \leq k \leq m$.  Here we use that the sequence $P_q$ converges to $P$ in order to ensure that the orders determined by $P_q$ and $P$ agree on the finite set of inequalities that determine $[g^k]_{P_q}$ and $[g^k]_P$.  Then by Lemma \ref{continuity}, $| [g^\ell]_{P_q} - [g^\ell]_P| < \frac{\ell}{m}+1$ for all $\ell >m$ and $q >N_m$.  Thus for $\ell, q > \max\{m, N_m\}$ we compute
\[\frac{| [g^\ell]_{P_q} - [g^\ell]_P|}{\ell} <  (\frac{\ell}{m}+1)\cdot \frac{1}{\ell} < \frac{2}{m} <\epsilon.
\]
This shows that 
$$\lim_{(q,k) \to \infty} \frac{[g^k]_{P_q}}{k}- \frac{[g^k]_{P}}{k} = 0.$$

Now considering the limit 
\[ \lim_{q \to \infty} \tau_{g}(P_q) - \tau_{g}(P) = \lim_{q \to \infty} \lim_{k \to \infty}\left( \frac{[g^k]_{P_q} - [g^k]_P}{k}\right)
\]
we observe that for each fixed $q$ the limit $\lim_{k \to \infty}\left( \frac{[g^k]_{P_q} - [g^k]_P}{k}\right)$ exists and is equal to $\tau_{g}(P_q) - \tau_{g}(P)$.  It follows (e.g. \cite[Corollary 19.6]{Bar}) that $ \lim_{q \to \infty} \tau_{g}(P_q) - \tau_{g}(P) =0$, proving the claim.
\end{proof}

We remark that some generalisations of this proposition are possible, though are not required for this manuscript:

\begin{remark}
\begin{enumerate}

\item It is possible to extend the map $\tau_{g}$ to have $\mathrm{LO}(G)$ as its domain.  To do this, we
observe that if $P$ does not contain $z$, i.e. if $z^{-1} \in P$, then we can set
\[ \tau_{g}(P) = - \lim_{k \to \infty} \frac{[g^k]_{P^{-1}}}{k}.
\]
Continuity of the map $\tau_{g}: \mathrm{LO}(G) \rightarrow \mathbb{R}$ defined in this way follows from Proposition \ref{continuity} and the fact that $ U_{z} \cup U_{z^{-1}}$ is a separation of $\mathrm{LO}(G)$.  

\item For a central element $z$ in a left-orderable group $G$, we can let $$U_z = \{P \in LO(G)\mid z \in P \text{ and $z$ is cofinal with respect to $P$}\}$$ (note that $U_z$ may be empty). The proof of Proposition \ref{continuity} shows that for $g \in G$, the map $\tau_g:U_z \to \mathbb{R}$ is continuous.
\item If the countability assumption on $G$ is removed, $LO(G)$ is no longer a sequential space \cite[Example 4.12]{BC21}. However, proofs above can be recast in the language of nets to remove the countability assumption. We will not need this level of generality for the applications that follow.
\end{enumerate}
\end{remark}

For the next lemma, note that if $P \in  \mathrm{LO}_{rel}(G)$ is a positive cone relative to the proper subgroup $C$, then for every $Q \in \mathrm{LO}(C)$ the set $Q \cup P$ is the positive cone of a left ordering of $G$---that is, $Q \cup P \in \mathrm{LO}(G)$. The corresponding ordering is the standard lexicographic one, constructed as follows: assume that the proper subgroup $C$ of $G$ admits a left order $<_Q$ corresponding to the positive cone $Q$, and that the left cosets $G/C$ admit a total ordering $<_P$ with corresponding relative cone $P$. 

We define the total ordering $<$ on $G$
as follows: 
\[
g_1< g_2 \Leftrightarrow \begin{cases}
g_1 C <_P g_2 C \ &\text{if } g_1^{-1}g_2\notin C,\\
id <_Q g_1^{-1}g_2 &\text{if } g_1^{-1}g_2\in C.
\end{cases}
\]

We require this construction for the next lemma.

\begin{lemma}
\label{relative_translation}
Suppose that $G$ is a group admitting an absolutely cofinal, central element $z \in G$. 
 Further suppose that $P \in \mathrm{LO}_{rel}(G)$ is a positive cone relative to the proper subgroup $C \subset G$,  that $z \in P$ and that $Q_1, Q_2 \in \mathrm{LO}(C)$.  Then for all $g \in G$ we have $\tau_{g}(Q_1 \cup P) = \tau_{g}(Q_2 \cup P)$.
\end{lemma}
\begin{proof} 
Set $P_1 = Q_1 \cup P$, $P_2 = Q_2 \cup P$, each with corresponding orderings $<_1$ and $<_2$ of $G$ constucted lexicographically from $<_{Q_1}$, $<_{Q_2}$ and $<_P$ as in the preceding paragraphs.

Fix $g \in G$, and suppose that 
\[z^{m} \leq_1 g <_1 z^{m+1},
\]
so that $[g]_{P_1} = m$.  

First note that if $z^{-m} g \notin C$ and $z^{-m-1} g \notin C$, then $z^{m}C <_P g C <_P z^{m+1} C$ and so $z^{m} <_i g <_i z^{m+1}$ for $i=1,2$. It follows that $[g]_{P_1} = [g]_{P_2}$.

Next, suppose that $z^{-m} g \in C$. Observe that this implies $z^{-m-1} g \notin C$ and $z^{-m+1} g \notin C$, because $z^{-m} g, z^{-m-1} g \in C$ implies $z = z^{m+1} g^{-1} g z^{-m} \in C$, a contradiction--similarly for $z^{-m+1} g$.  Therefore from the inequality $z^{zm} \leq_1 g <_1 z^{m+1}$, which implies $z^{m-1} <_1 g <_1 z^{m+1}$, we deduce that $z^{m-1}C <_P g C<_P z^{m+1}C$ and so $z^{m-1} <_2 g <_2 z^{m+1}$.  Therefore $[g]_{P_2} \in \{m-1, m\}$.

Last, suppose $z^{-m-1} g \in C$ and observe that this implies $z^{-m} g \notin C$ and $z^{-m-2} g \notin C$.  Proceeding as in the previous paragraph we conclude that $[g]_{P_2} \in \{m, m+1\}$.

We conclude that $|[g]_{P_1} - [g]_{P_2}| \leq 1$ for all $g \in G$, so that for any $g \in G$
\[ \lim_{k \to \infty} \frac{[g^k]_{P_1} - [g^k]_{P_2}}{k} = 0,
\]
so that $\tau_{g}(Q_1 \cup P) = \tau_{g}(P_1) = \tau_{g}(P_2) = \tau_{g}(Q_2 \cup P)$.
\end{proof}

\section{The Dehornoy ordering, isolated points and accumulation points}
\label{sec:dehornoy}
In this section we use the algebraic properties of the Dehornoy ordering, and the closely associated Dubrovina-Dubrovin ordering, to prove our main theorem.  We frequently use, without mention, that $\Delta_n^2 \in B_n$ is absolutely cofinal \cite[Chapter II, Proposition 3.6]{DDRW}. 

\subsection{Generalities on ordered sets and group actions}

Suppose that $ (\Omega, <)$ is a totally ordered set, and that $G$ is a left-ordered group acting on $\Omega$ in an order-preserving fashion.  For each $x \in \Omega$, and for each choice of positive cone $P \in \mathrm{LO}(\Stab(x))$, one can create a positive cone $Q_{P,x} \subset G$ according to the rule:
\[ g \in Q_{P,x}  \iff g\cdot x >x \mbox{ or } g \in \Stab(x) \mbox{ and } g \in P.
\]
For each $x \in \Omega$, set 
\[ \mathfrak{O}_x = \{ Q_{P,x} \mid P \in \mathrm{LO}(\Stab(x)) \},
\]
and for each subset $S \subset \Omega$, set 
\[\mathfrak{O}_S = \bigcup_{x \in S} \mathfrak{O}_x.
\]
\begin{proposition} \cite[\textit{c.f.} Lemmas 3.6 and 3.7]{BC} 
\label{conv seq}
Let $G$ be a countable left-ordered group acting on a totally ordered set $ (\Omega, <)$ by order-preserving bijections.  Let $x \in \Omega$, and suppose that there exists a sequence of points $\{x_i\}_{i=1}^{\infty} \subset \Omega$ converging to $x$ in the order topology on $\Omega$. Let $S \subset \Omega$.  If $\{x_i\}_{i=1}^{\infty} \subset S$ then $\mathfrak{O}_x \cap \overline{\mathfrak{O}_S} \neq \emptyset$.
\end{proposition}
\begin{proof}

First, for each $i \in \mathbb{N}$ choose a positive cone $Q_i \in \mathfrak{O}_{x_i}$, and set $P_i = Q_i \cap \Stab(x)$.   Then $\{P_i\}_{i=1}^{\infty}$ is a sequence of positive cones in $\mathrm{LO}(\Stab(x))$, choose a convergent subsequence $\{P_{i_j}\}_{j=1}^{\infty}$ with limit $P \in \mathrm{LO}(\Stab(x))$--this is possible since $G$ is countable and therefore $\mathrm{LO}(\Stab(x))$ is a compact metric space.\footnote{Countability is necessary here to ensure that $\mathrm{LO}(G)$ is sequentially compact, which does not hold in general.}

We claim that $Q_{P, x} \in \overline{\mathfrak{O}_S}$ since $\{Q_{i_j}\}_{j=1}^{\infty}$ converges to $Q_{P,x}$, which will complete the proof. To see this, suppose $\bigcap_{i=1}^n U_{g_i}$ is an arbitrary basic open neighbourhood of $Q_{P, x}$ in $LO(G)$.  Then for each $i=1, \dots, n$, either $g_i \cdot x >x$ or $g_i \in \Stab(x) \cap P$; suppose without loss of generality that $g_i \cdot x >x$ for $i = 1, \dots, k$ and that $g_i \in \Stab(x) \cap P$ for $i = k+1, \dots, n$.

Set $y = \min_{i = 1, \dots, k}\{ g_i \cdot x \}$ and $y' = \max_{i=1, \dots, k} \{g_i^{-1} \cdot x\}$, so that $(y', y)$ is a neighbourhood of $x$ in the order topology on $\Omega$.  Now choose $s, t \in \mathbb{N}$ as follows:  Choose $s$ large enough that for all $j>s$, $x_{i_j} \in  (y', y)$, and choose $t$ large enough that for all $j >t$ we have $P_{i_j} \in \bigcap_{i=k+1}^n U_{g_i} \subset \mathrm{LO}(\Stab(x))$.  Set $N = \max\{s, t\}$.

We finish by showing that $Q_{i_j} \in \bigcap_{i=1}^n U_{g_i}$ for $j>N$.  First, $Q_{i_j} \in \bigcap_{i=1}^k U_{g_i}$, which we can verify by considering cases.  Fix $g_i$ with $i \in \{1, \dots, k\}$.

\noindent \textbf{Case 1.} $x \leq x_{i_j} <y$.  Then  $g_i \cdot x_{i_j} \geq y > x_{i_j}$ so that $Q_{i_j} \in U_{g_i}$.

\noindent \textbf{Case 2.} $y'  < x_{i_j} < x$.   Then  $g_i^{-1} \cdot x_{i_j}  \leq y' < x_{i_j} $ so that $x_{i_j} < g \cdot x_{i_j}$ and again $Q_{i_j} \in U_{g_i}$.


Next, $Q_{i_j} \in \bigcap_{i=k+1}^n U_{g_i}$ because $g_i \in P_{i_j} \subset Q_{i_j}$ for all $i \in \{k+1, \dots, n\}$ and for all $j >t$.  Overall, this yields $Q_{i_j} \in \bigcap_{i=1}^n U_{g_i}$  and completes the claim.

\end{proof}

\subsection{The Dehornoy and Dubrovina-Dubrovin orderings}

We recall the definition of the Dehornoy ordering.  Given a word $w$ in the generators $\sigma_1, \dots, \sigma_{n-1}$, we call $w$ an $i$-positive word (resp. $i$-negative word) if there exists $i \in \{ 1, \dots, n-1\}$ such that $\sigma_i$ occurs in $w$ with only positive exponents (resp. negative exponents), and there are no occurrences of $\sigma_j^{\pm 1}$ with $j>i$.  We call a braid $\beta \in B_n$ an $i$-positive braid if $\beta$ admits an $i$-positive representative word.   Set
\[ P_D = \{ \beta \in B_n \mid \beta \mbox{ is $i$-positive for some $1 \leq i \leq n-1$} \}.
\]

\begin{theorem}[\cite{Deh}]
The set $P_D$ is the positive cone of a left ordering of $B_n$.
\end{theorem}

As a consequence, 
\[ P_D^{-1} = \{ \beta \in B_n \mid \beta \mbox{ is $i$-negative for some $1 \leq i \leq n-1$} \}.
\]

Note that this definition is slightly different than the typical definition of the Dehornoy ordering, where one would ask that there are no occurrences of $\sigma_j^{\pm 1}$ with $j<i$ \cite{DDRW}.   Our reason for doing this is as follows.  Recall that for $m,n \in \mathbb{N}$ and $m<n$, there is an embedding $i:B_m \rightarrow B_n$ given by $i(\sigma_j) = \sigma_j$ for all $1 \leq j \leq m-1$.  For ease of notation, we will write $B_m$ in place of $i(B_m)$.

\begin{proposition}\cite[Chapter II, Proposition 3.16]{DDRW}
Given $n, m >1$ with $m<n$, the subgroup $B_m$ is convex relative to the Dehornoy ordering of $B_n$.
\end{proposition}
\begin{proof} Let $<_D$ be the ordering corresponding to $P_D$.
Without loss of generality assume that $id <_D \alpha <_D \beta$ for some $\alpha \in B_n$ and $\beta\in B_m$. Hence $\alpha$ is $i$-positive and $\beta$ is $j$-positive for some $i\in\{1,\dots, n-1\}$ and $j\in \{1,\dots, m-1\}$.  

Assume $i> m-1$, we will arrive at a contradiction.  Let $w$ be an $i$-positive representative word for $\alpha$, and $v$ a $j$-positive representative word for $\beta$.  Then $w^{-1}v$ is a representative word for $\alpha^{-1} \beta$, which is $i$-negative since it contains only occurrences of $\sigma_i^{-1}$ while containing no occurrences of $\sigma_k$ for $k>i$.  But then $\alpha^{-1} \beta \in P_D^{-1}$, a contradiction to $\alpha <_D \beta$.
\end{proof}

For ease of exposition, we single out a particular subset of the positive cone of the Dehornoy ordering.  We set
$$P_D^{rel} = \{ \beta \in B_n \mid \mbox{ $\beta$ is $(n-1)$-positive} \} \subset B_n,$$
note that the definition of $P_D^{rel}$ depends on the index $n$, so we will always take care to specify the index when we are working with this subset.

\begin{proposition}
\label{PDrel}
The subset $P_D^{rel} \subset B_n$ and its inverse $(P_D^{rel})^{-1}$ are positive cones relative to $B_{n-1}$.
\end{proposition}
\begin{proof}
We verify the three conditions in the definition of a relative positive cone.

$(i)$ Since $\sigma_{n-1}\in P_D^{rel}$ we know $P_D^{rel}\neq \{id\}$.

$(ii)$ Let $\alpha, \beta \in B_{n-1}$ and $\gamma\in P_D^{rel}$. Then $\gamma$ is $(n-1)$-positive and $\alpha, \beta$ are represented by words which contain no occurrences of $\sigma_{n-1}$, which implies that $\alpha\gamma\beta$ is also $(n-1)$-positive. 

$(iii)$ Since $P_D$ is the positive cone of a left ordering and $P_D^{rel} \subset P_D$, we conclude that $P_D^{rel}$ and $(P_D^{rel})^{-1}$ are disjoint.  Neither contains elements of $B_{n-1} \subset B_n$, by definition.  This shows that $B_{n-1} \cup P_D^{rel} \cup(P_D^{rel})^{-1}$ is a disjoint union; that it is equal to $B_n$ follows from the fact that every element of $B_n$ is $i$-positive for some $i \in \{1, \ldots, n-1\}$, and those that are $i$-positive for $i \in \{1, \ldots, n-2\}$ are the elements of $B_{n-1}$.
\end{proof}

In preparation for the next lemma, we introduce another canonical ordering of $B_n$, known as the Dubrovina-Dubrovin ordering.  We define its positive cone, denoted $P^n_{DD} \subset B_n$, inductively. 

For $B_2 = \langle \sigma_1 \rangle \cong \mathbb{Z}$, set $P^2_{DD} = \{ \sigma_1^k \mid k>1 \}$.  Now, assuming that we have defined $P^{n-1}_{DD} \subset B_{n-1}$, define $P^{n}_{DD}$ according to the rule:
\[ P^{n}_{DD} = P^{n-1}_{DD} \cup (P_D^{rel})^{(-1)^n},
\]
where $P_D^{rel} \subset B_n$.  Note that this defines a positive cone by Proposition \ref{PDrel}.  This positive cone is of interest as it is an isolated point in $\mathrm{LO}(B_n)$, specifically:

\begin{theorem}[\cite{DD}]
\label{DDtheorem}
If $Q \subset B_n$ is a positive cone containing the braids 
\[ \sigma_1, (\sigma_1 \sigma_2)^{-1}, (\sigma_1 \sigma_2 \sigma_3), \dots, (\sigma_1 \dots \sigma_{n-1})^{(-1)^n}
\]
then $Q = P^n_{DD}$.
\end{theorem}

We remark that this definition of the positive cone differs from that in \cite{DD}.  In the original source, the positive cone $P^n_{DD}$ is defined to be the subsemigroup of $B_n$ generated by 
\[ (\sigma_1 \sigma_2 \dots \sigma_{n-1}), (\sigma_1 \sigma_2 \dots \sigma_{n-2})^{-1} ,\dots, (\sigma_{n-1})^{(-1)^n},
\]
we have altered the definition here for ease of exposition, so that the convex subgroups agree with those arising from our definition of $P_D$.  One can see that our definition of $P^n_{DD}$ is equivalent to that of \cite{DD} by noting that our $P^n_{DD}$ is the image of Dubrovina and Dubrovin's positive cone under the automorphism $\phi : B_n \rightarrow B_n$ given by $\phi(\sigma_i) = (\sigma_{n-i})^{(-1)^n}$.

\begin{lemma}
\label{relative rigidity}
Suppose that $Q \subset B_n$ is a positive cone relative to $B_{n-1}$.  Then $Q \in \{P_D^{rel}, (P_D^{rel})^{-1} \}$.
\end{lemma}
\begin{proof}
Let $Q \subset B_n$ be a positive cone relative to $B_{n-1}$. Note that $\sigma_1 \sigma_2 \dots \sigma_{n-1} \notin B_{n-1}$, and so since $B_n = B_{n-1} \cup Q \cup Q^{-1}$, we can choose $\epsilon \in \{ \pm 1\}$ so that $(\sigma_1 \sigma_2 \dots \sigma_{n-1})^{(-1)^n} \in Q^{\epsilon}$.  Then 
$Q' = P^{n-1}_{DD} \cup Q^{\epsilon}
$
is the positive cone of a left ordering, and by our choice of $\epsilon$, $Q'$ contains $$\sigma_1, (\sigma_1 \sigma_2)^{-1}, (\sigma_1 \sigma_2 \sigma_3), \dots, (\sigma_1 \dots \sigma_{n-1})^{(-1)^n}.$$ This implies that $Q' = P^n_{DD}$ by  Theorem \ref{DDtheorem}.  In particular, this means that 
\[Q^{\epsilon} = Q' \cap (B_n \setminus B_{n-1}) = P^n_{DD} \cap (B_n \setminus B_{n-1}) = (P_D^{rel})^{(-1)^n}.
\]
 \end{proof}

\begin{corollary}
\label{maximal convex}
Suppose that $C$ is a relatively convex subgroup satisfying $B_{n-1} \subseteq C \subseteq B_n$.  Then $C = B_{n-1}$ or $C = B_n$.
\end{corollary}
\begin{proof} 
Assume that $B_{n-1} \subset C \subset B_n$, with both containments proper. Choose a positive cone $P \in \mathrm{LO}(B_n)$ such that $C$ is convex relative to the induced ordering $<_P$.  Set $Q = P \cap (B_n\setminus C)$, so that $Q$ is a positive cone relative to $C$.  Since $B_{n-1} \subset C$ and $C \neq B_{n}$, we know that $(\sigma_1 \sigma_2 \dots \sigma_{n-1})^{(-1)^n} \in Q \cup Q^{-1}$.  Suppose $(\sigma_1 \sigma_2 \dots \sigma_{n-1})^{(-1)^n} \in Q,$ the case of $Q^{-1}$ being similar.

Next, set $R = P^{n}_{DD} \cap (C \setminus B_{n-1})$, which is a positive cone in $C$ relative to $B_{n-1}$, as is $R^{-1}$.  Now we observe that $P^{n-1}_{DD} \cup R \cup Q$ and $P^{n-1}_{DD} \cup R^{-1} \cup Q$ are distinct positive cones in $B_n$, each containing the elements 
\[ \sigma_1, (\sigma_1 \sigma_2)^{-1}, (\sigma_1 \sigma_2 \sigma_3), \dots, (\sigma_1 \dots \sigma_{n-1})^{(-1)^n}
\]
by construction.  This violates Theorem \ref{DDtheorem}, and so is a contradiction.
\end{proof}

Note that the embeddings $B_m \subset B_n$ for $m<n$ also allow us to think of $\Delta_m^2$ as an element of $B_n$ for all $m<n$.  We can apply the previous Corollary to analyse actions of $B_n$ by order-preserving automorphisms of an arbitrary totally-ordered set.

\begin{proposition}
\label{stab_prop}
Suppose that $B_n$ acts nontrivially (i.e. without global fixed points) on a totally ordered set $(\Omega, <)$ by order-preserving bijections. If $\Delta_{n-1}^2 \in \Stab(x_0)$ for some $x_0 \in \Omega$, then $\Stab(x_0) = B_{n-1}$.
\end{proposition}
\begin{proof}
Assuming $\Delta_{n-1}^2 \in \Stab(x_0)$ for some $x_0 \in \Omega$, equip $B_n$ with a left ordering $\prec$ relative to which $\Stab(x_0)$ is convex as follows. Choose $Q \in \mathrm{LO}(\Stab(x_0))$, without loss of generality we may assume that $\Delta_{n-1}^2 \in Q$, then declare $id \prec \beta$ if and only if $\beta \cdot x_0 > x_0$ or $\beta \in \Stab(x_0)$ and $\beta \in Q$.  

Now observe that if $\alpha \in B_{n-1} \setminus \Stab(x_0)$ and $id \prec \alpha$, then $id \prec \Delta_{n-1}^{2k} \prec \alpha$ for all $k >0$ since $\Delta_{n-1}^2$ is contained in the $\prec$-convex subgroup $\Stab(x_0)$, while $\alpha$ is not.  This contradicts the fact that $\Delta_{n-1}^2$ is cofinal in every left ordering of $B_{n-1}$, so we conclude $B_{n-1} \subset \Stab(x_0)$.

Thus $B_{n-1} \subset \Stab(x_0) \subset B_n$.  Since we are assuming the second containment is proper (the action is without global fixed points), Corollary \ref{maximal convex} implies $B_{n-1} = \Stab(x_0)$.
\end{proof}

\begin{theorem}
\label{main_theorem}
Suppose that $P \in \mathrm{LO}(B_n)$ and $\Delta_{n-1}^2$ is not cofinal in the left ordering of $B_n$ determined by $P$.  Then:
\begin{enumerate}
\item If $\Delta_n^2 \in P$, then there exists $Q \in \mathrm{LO}(B_{n-1})$ such that $$P_D^{rel} \cup Q \in \overline{\{\beta P \beta^{-1} \mid \beta \in B_n \}},$$
\item If $\Delta_n^2 \in P^{-1}$, then  there exists $Q \in \mathrm{LO}(B_{n-1})$ such that $$(P_D^{rel})^{-1} \cup Q \in \overline{\{\beta P \beta^{-1} \mid \beta \in B_n \}}.$$
\end{enumerate}
\end{theorem}
\begin{proof}
We show (1). Let $<$ denote the left ordering determined by $P$ relative to which $\Delta_{n-1}^2$ is not cofinal, and note that $\Delta_n^2$ is positive.  Consider the totally ordered set 
\[ S = \{ T \subset B_n \mid \beta \in T \mbox{ and } \alpha < \beta \mbox{ implies } \alpha \in T \}.
\] 
This set is totally ordered by inclusion, and one can verify that left-multiplication by elements of $B_n$ defines an action of $B_n$ on $S$ by order-preserving bijections.  

Consider the element $s_0 \in S$ defined by 
\[s_0 = \{ \beta \in B_n \mid \beta < \Delta_{n-1}^{2k} \mbox{ for some } k \in \mathbb{Z} \}. 
\]
Without loss of generality we assume that $id<\Delta_{n-1}^2$, and define a sequence $\{s_i\}_{i=1}^{\infty} \subset S$ by  $s_i = \{ \beta \in B_n \mid \beta \leq \Delta_{n-1}^{2i} \}$ for $i \geq 1$ (if $id > \Delta_{n-1}^2$, define an analogous sequence using negative powers).  We show that $\{s_i\}_{i=1}^{\infty}$ converges to $s_0$ in the order topology on $S$.

To see this, let $(t_0, t_1)$ be an open interval containing $s_0$.  Then the containment $t_0 \subset s_0$ is proper, so there exists $\beta \in s_0 \setminus t_0$.  Since $\beta \in s_0$ there exists $N>0$ such that $\beta < \Delta_{n-1}^{2k}$ for all $k>N$, so that $\beta \in s_{k}$.  Since every element in $t_0$ is less than $\beta$, this shows $t_0 \subset s_k \subset s_0$ and so $s_k \in (t_0, t_1)$ for all $k>N$.

Now we are in a position to apply Proposition \ref{conv seq}.  For each $\beta \in B_n$, define the set $t_{\beta} = \{ \alpha \in B_n \mid \alpha \leq \beta\}$, and set $T  = \{ t_{\beta} \mid \beta \in B_n\} \subset S$ and note $\{s_i\}_{i=1}^{\infty} \subset T \subset S$.   The stabiliser of each $t_{\beta} \in S$ under the $B_n$-action is trivial, and $\mathfrak{O}_{t_{\beta}} = \{ \beta P \beta^{-1} \}$ for each $\beta \in B_n$.  Thus $\mathfrak{O}_T = \{\beta P \beta^{-1} \mid \beta \in B_n \}$ and by Proposition \ref{conv seq}, $\mathfrak{O}_{s_0} \cap  \overline{\{\beta P \beta^{-1} \mid \beta \in B_n \}} \neq \emptyset$.

It remains to show that every element of $\mathfrak{O}_{s_0}$ is of the form $P_D^{rel} \cup Q$ where $Q \in \mathrm{LO}(B_{n-1})$.  Every element of $\mathfrak{O}_{s_0}$ is of the form $P \cup Q$ where $P$ is a positive cone in $B_n$ relative to $\Stab(s_0)$, and $Q \in \mathrm{LO}(\Stab(s_0))$.  By Proposition \ref{stab_prop} we have $\Stab(s_0) = B_{n-1}$, so $Q \in \mathrm{LO}(B_{n-1})$.  Now by Lemma \ref{relative rigidity}, $P \in \{ P_D^{rel}, (P_D^{rel})^{-1} \}$.  Since we have assumed $\Delta_n^2 \in P$, we also have $\Delta_n^2 \in \beta P \beta^{-1}$ for all $\beta \in B_n$ and thus $\Delta_n^2$ is contained in every limit point of $\{\beta P \beta^{-1} \mid \beta \in B_n \}$.   It follows that $P = P_D^{rel}$, finishing the proof.  

The case where $\Delta_n^2 \in P^{-1}$ is nearly identical in its proof.
\end{proof}

The next theorem shows that whether or not $\Delta_{n-1}^2$ is cofinal in an ordering completely determines the translation numbers of $\beta \in B_n$ relative to that ordering.

\begin{theorem}
\label{translation_number_same}
Suppose that $P \in \mathrm{LO}(B_n)$ and $\Delta_n^2 \in P$,  and that $\Delta_{n-1}^2$ is not cofinal in the left ordering of $B_n$ determined by $P$.  Then $\tau_{\beta}(P) = \tau_{\beta}(P_D)$ for all $\beta \in B_n$.
\end{theorem}
\begin{proof}
By Theorem \ref{main_theorem}, there exists $Q \in \mathrm{LO}(B_{n-1})$ such that $$P_D^{rel} \cup Q \in \overline{\{\alpha P \alpha^{-1} \mid \alpha \in B_n \}}.$$

Fix $\beta \in B_n$.  By Lemma \ref{conjugation}, $\tau_{\beta}$ is constant on $\{\alpha P \alpha^{-1} \mid \alpha \in B_n \}$, and by Lemma \ref{continuity}, it is therefore constant on $ \overline{\{\alpha P \alpha^{-1} \mid \alpha \in B_n \}}$.  In particular, we conclude that $\tau_{\beta}(P)  = \tau_{\beta}( P_D^{rel} \cup Q)$.  By Lemma \ref{relative_translation}, $\tau_{\beta}(P_D^{rel} \cup Q) = \tau_{\beta}(P_D)$, which concludes the proof.

\end{proof}

From this, we arrive at the following corollary, which we can think of as meaning that $\Delta_{n-1}^2$ serves as a ``test element'' for computations of the translation numbers defined by a given left ordering of $B_n$. 

\begin{corollary}
\label{cor:agreewithDehornoy}
Suppose that $P \in \mathrm{LO}(B_n)$ and $\Delta_n^2 \in P$.  Then $\tau_{\beta}(P) = \tau_{\beta}(P_D)$ for all $\beta \in B_n$ if and only if $\tau_{\Delta_{n-1}^2}(P) = \tau_{\Delta_{n-1}^2}(P_D) = 0$.
\end{corollary}
\begin{proof}
Observe that if $\tau_{\Delta_{n-1}^2}(P) = \tau_{\Delta_{n-1}^2}(P_D)$, then $\tau_{\Delta_{n-1}^2}(P) = 0$ as $\Delta_{n-1}^2$ is bounded in the Dehornoy ordering (it is contained in the convex subgroup $B_{n-1}$).

On the other hand, if $\tau_{\Delta_{n-1}^2}(P) = 0$ then we suppose that $\Delta_{n-1}^{2m} > \Delta_{n}^2 $ for some $m>0$ and arrive at a contradiction (the case of $m<0$ being identical).  For in this case, $\Delta_{n-1}^{2km} > \Delta_{n}^{2k}$ for all $k>0$, so that $[\Delta_{n-1}^{2km}]_P> k$.  We compute
\[\tau_{\Delta_{n-1}^2}(P) = \lim_{k \to \infty} \frac{[\Delta_{n-1}^{2km}]_P}{km} \geq \lim_{k \to \infty} \frac{k}{km} = \frac{1}{m},
\]
this contradiction shows $\Delta_{n-1}^{2m}$ must be bounded above by $\Delta_{n}^2$.  Now apply Theorem \ref{translation_number_same}.
\end{proof}

\subsection{Many orderings sharing the same translation numbers}

At this point, it is a natural question to ask how many left orderings of $B_n$ satisfy the condition that $\Delta_{n-1}^2$ is bounded.  We answer this question with the following proposition:

\begin{proposition}
There are uncountably many left orderings of $B_n$ relative to which $\Delta_{n-1}^2$ is bounded.
\end{proposition}
\begin{proof}
Consider the subgroup $H = \langle \Delta_{n-1}^2, \Delta_n^2 \rangle$, which is isomorphic to $\mathbb{Z} \oplus \mathbb{Z}$, and the restriction map $r : \mathrm{LO}(B_n) \rightarrow \mathrm{LO}(H)$ given by $r(P) = P \cap H$.  This map is clearly continuous.

Considering the short exact sequence
\[ 0 \rightarrow \langle \Delta_{n-1}^2 \rangle \rightarrow H \rightarrow \langle \Delta_n^2 \rangle \rightarrow 0,
\]
we let $S \subset \mathrm{LO}(H)$ denote the set consisting of the four lexicographic orderings arising from this short exact sequence. Observe that $P \in \mathrm{LO}(B_n)$ corresponds to an ordering where $\Delta_{n-1}^2$ is bounded if and only if $r(P)\in S$.  

Next, observe that $\beta P_D \beta^{-1} \in r^{-1}(S)$ for all $\beta \in B_n$.  Moreover, the Dehornoy ordering is an accumulation point of its own conjugates \cite[Chapter XIV, Proposition 2.3]{DDRW}.  Therefore the set $\overline{\{ \beta P_D \beta^{-1} \mid \beta \in B_n \}}$ is a closed (hence compact) subset of $r^{-1}(S)$ which contains no isolated points.  Thus it is uncountable, meaning $r^{-1}(S)$ is uncountable.
\end{proof}

\section{Fractional Dehn twist coefficients of low-genus surfaces}

\label{FDTC}

Let $\Sigma_{g,n}^b$ denote a surface of genus $g$ with $n$ marked points and $b$ boundary components.  Given a curve $\alpha \in \Sigma_{g,n}^b$ we use $T_\alpha \in  \Mod(\Sigma_{g,n}^b)$ to denote the class of a Dehn twist about $\alpha$; given a boundary component $d$ we use $T_d$ to denote the class of a Dehn twist about a curve parallel to the boundary component $d$. 

Now, if $\Sigma_{g,n}^b$ is a hyperbolic surface with boundary (i.e. $b>0$, and $g > 0$ or $n > 1$), then we can construct a particular action of $\Mod(\Sigma_{g,n}^b)$ on $\mathbb{R}$ as below, and use it to define the fractional Dehn twist coefficient of an element $h\in  \Mod(\Sigma_{g,n}^b)$.

Let $p: \widetilde \Sigma_{g,n}^b \rightarrow \Sigma_{g,n}^b$ denote the universal cover, and fix an isometric embedding $\widetilde \Sigma_{g,n}^b \subset \mathbb{H}^2$.   Choose $x_0 \in \partial \Sigma_{g,n}^b$ and let $d$ denote the component of the boundary fixing $x_0$.  Fix a choice of lift $\tilde{x_0} \in \tilde{d} \subset \partial \widetilde \Sigma_{g,n}^b$, so that for each $h \in \Mod(\Sigma_{g,n}^b)$ there is a unique lift of $h:\widetilde \Sigma_{g,n}^b \rightarrow \widetilde \Sigma_{g,n}^b$ satisfying $h(\tilde{x_0}) = \tilde{x_0}$, so that $h$ also fixes $\tilde{d}$.  Thus we have an action of $ \Mod(\Sigma_{g,n}^b)$ on $\widetilde \Sigma_{g,n}^b$ fixing $\tilde{d}$, which extends to an action on $\partial \widetilde \Sigma_{g,n}^b$ by orientation-preserving homeomorphisms.

 Let $\ell \subset \overline{\mathbb{H}^2}$ denote the closure of $\tilde{d}$, and identify $\partial \widetilde \Sigma_{g,n}^b \setminus \ell$ with the interval $(0,\pi)$, and thus with $\mathbb{R}$, by identifying each point $y$ on the boundary with the unique geodesic from $x_0$ to $y$.  Now the action on the boundary of $\Mod(\Sigma_{g,n}^b)$ yields and orientation-preserving action on $\mathbb{R}$.  We orient the boundary and parameterise it so that the action of the boundary Dehn twist $T_d$ satisfies $T_d(x) = x+1$ for all $x \in \mathbb{R}$.  This gives 
 \[ \rho_{s,d}: \Mod(\Sigma_{g,n}^b) \rightarrow \mathrm{H}\widetilde{\mathrm{ome}}\mathrm{o}_+(S^1)
 \]
which we call the \emph{standard representation with respect to $d$}, or the \emph{Nielsen-Thurston action on $\mathbb{R}$ with respect to $d$}.  The \emph{fractional Dehn twist coefficient} of $h \in \Mod(\Sigma_{g,n}^b)$ can be defined as
\[ c(h, d) = \tau_h(\rho_{s,d}).
\]
That this is equivalent to the usual definition of the fractional Dehn twist coefficient can be found in \cite[Theorem 4.16]{IK}, and for the special case of the braid groups (which is our case of interest) the result appears in \cite{Mal}.  When  there is a single boundary component (as in our case of interest), we do not need to keep track of the component $d$ used in this construction, so we simply use the notation $ \rho_{s}$ and $ c(h)$.

\begin{theorem}
\label{thm:mappingclassgroupactions}
Suppose that $\rho: \mathrm{Mod}(\Sigma_{0,n}^1) \rightarrow \mathrm{Homeo_+}(\mathbb{R})$ is an injective homomorphism such that the induced action of $\mathrm{Mod}(\Sigma_{0,n}^1) $ on $\mathbb{R}$ is without global fixed points, and let $d$ denote a simple closed curve parallel to the boundary component of  $\Sigma_{0,n}^1$.  If there exists a simple closed curve $\alpha \subset  \Sigma_{0,n}^1 $ such that: 
\begin{enumerate}
\item  $\Sigma_{0,n}^1 \setminus \{ \alpha \} \cong \Sigma_{0,n-1}^1 \cup \Sigma_{0,1}^2$, and
\item $\rho(T_{\alpha})$ has a fixed point in $\mathbb{R}$,
\end{enumerate}
then $\rho$ is conjugate by a (possibly orientation-reversing) homeomorphism of $\mathbb{R}$ to a representation $\rho' :  \mathrm{Mod}(\Sigma_{0,n}^1) \rightarrow \mathrm{H\widetilde{ome}o}_+(\mathbb{S}^1)$ such that:
\begin{enumerate}
\item   $\rho'(T_d)(x) = x+1$ for all $x \in \mathbb{R}$, and
\item $c(h) =\tau_h(\rho')$ for all $h \in  \mathrm{Mod}(\Sigma_{0,n}^1)$.
\end{enumerate}

\end{theorem}
\begin{proof} 
Let $\phi:B_n \rightarrow \mathrm{Mod}(\Sigma_{0,n}^1)$ denote the standard isomorphism.  
First note that as $\rho$ provides an action without global fixed points, the element $\rho(\phi(\Delta_n^2))= \rho(T_d)$ must act without fixed points.  Thus, by using an orientation-reversing homeomorphism if necessary, we may conjugate $\rho$ to produce $\rho'$ with image in $\mathrm{H\widetilde{ome}o}_+(\mathbb{S}^1)$.

Note also that $\phi(\Delta_{n-1})$ 
 is represented by a Dehn twist about a curve $\gamma$ such that $\Sigma_{0,n}^1 \setminus \{ \gamma \} \cong \Sigma_{0,n-1}^1 \cup \Sigma_{0,1}^2$.  Moreover, if $\alpha$ is any other curve such that $\Sigma_{0,n}^1 \setminus \{ \alpha \} \cong \Sigma_{0,n-1}^1 \cup \Sigma_{0,1}^2$, then there exists a homeomorphism of $\Sigma_{0,n}^1$ carrying $\gamma$ to $\alpha$ \cite[Section 1.3]{FM}.  Thus $\rho'(\phi(\Delta_{n-1}))$ is conjugate to an element having a fixed point, so itself has a fixed point.

Now fix an enumeration $\{r_0, r_1, \dots \}$ of the rationals, and define a positive cone $P \subset B_n$ in the usual way:  $\beta \in B_n$ satisfies $\beta \in P$ if and only if $\rho'(\phi(\beta))(r_i) >r_i$, where $i$ is the smallest index such that $\rho'(\phi(\beta))(r_i) \neq r_i$.  One can verify that $\tau_\beta(P) = \tau_\beta(\rho'\circ \phi)$ for all $\beta \in B_n$ and that $\tau_\beta(\Delta_{n-1}^2) (P) = 0$.  By Corollary \ref{cor:agreewithDehornoy} this implies that $\tau_\beta(P) = \tau_\beta(P_D)$ for all $\beta \in B_n$, which by \cite[Theorem 7.5]{Mal} implies that $\tau_\beta(P) = c(\phi(\beta))$, completing the proof.
\end{proof}

We can also offer a restatement of this theorem in terms of left orderings of $B_n \cong  \mathrm{Mod}(\Sigma_{0,n}^1)$, which appears as Theorem \ref{thm:mappingclassorders} in the introduction.

\begin{proof}[Proof of Theorem \ref{thm:mappingclassorders}]
Let $\rho_P$ denote the dynamic realisation of the given ordering.  The restriction $T_d > id$ guarantees that $\rho_P$ is conjugate via an orientation-preserving homeomorphism of $\mathbb{R}$ to a representation $\rho'$ with image in $\mathrm{H\widetilde{ome}o}_+(\mathbb{S}^1)$. Therefore $\tau_h(P) = \tau_{h}(\rho') = c(h)$ for all $h \in  \mathrm{Mod}(\Sigma_{0,n}^1)$ by Theorem \ref{thm:mappingclassgroupactions}
.
\end{proof}

We arrive at similar results for the surface $\Sigma_{1,0}^1$.  Recall that $\mathrm{Mod}(\Sigma_{1,0}^1) \cong B_3$, where the square of generator of the centre $\Delta_3^4$ is carried to the Dehn twist $T_d$ about the boundary curve $d$, and each generator is carried to a Dehn twist about a non-separating simple closed curve (e.g. see \cite{MWi}).  Therefore we conclude:

\begin{theorem}
\label{genus1}
Suppose that $\rho: \mathrm{Mod}(\Sigma_{1,0}^1) \rightarrow \mathrm{Homeo_+}(\mathbb{R})$ is an injective homomorphism such that the induced action of $\mathrm{Mod}(\Sigma_{1,0}^1) $ on $\mathbb{R}$ is without global fixed points, and let $d$ denote a simple closed curve parallel to the boundary component of  $\Sigma_{1,0}^1$.  Let $\alpha$ be a non-separating simple closed curve on $\Sigma_{1,0}^1$. If $\tau_{T_\alpha}(\rho) = 0$ 
then $\rho$ is conjugate by a (possibly orientation-reversing) homeomorphism of $\mathbb{R}$ to a representation $\rho' :  \mathrm{Mod}(\Sigma_{1,0}^1) \rightarrow \mathrm{H\widetilde{ome}o}_+(\mathbb{S}^1)$ such that:
\begin{enumerate}
\item   $\rho'(T_d)(x) = x+1$ for all $x \in \mathbb{R}$, and
\item $c(h) = \tau_h(\rho')$ for all $h \in  \mathrm{Mod}(\Sigma_{1,0}^1)$.
\end{enumerate}
\end{theorem}
\begin{proof}
We need only observe that when $n=3$, $\Delta_{n-1}^2 = \sigma_1^2$, and that if $\alpha, \gamma$ are nonseparating curves in $\Sigma_{1,0}^1$ then there exists a homeomorphism of $\Sigma_{1,0}^1$ carrying one to the other \cite[Section 1.3.1]{FM}.  Now proceed as in the proof of Theorem \ref{thm:mappingclassgroupactions}.
\end{proof}

In analogy with Theorem \ref{thm:mappingclassorders}, if so inclined, one can restate Theorem \ref{genus1} purely in terms of left orderings of the mapping class group $\mathrm{Mod}(\Sigma_{1,0}^1)$.


\section{Mapping class groups of low-genus surfaces acting on the circle}
\label{sec:MCG}

We use the results of the previous section to control the action of low-genus mapping class groups on $S^1$.  First, we require a brief background on circular orderings, their relationship with left orderings, and cohomology.

\subsection{Circular orderings and cohomology}
\label{subsec:lifting_basics}

Given a group $G$, recall that an inhomogeneous 2-cocyle $f: G^2 \rightarrow \{ 0, 1\}$ is a function satisfying $f(id, g) = f(g, id) = 0$ for all $g \in G$ and $f(h,k) - f(gh, k) + f(g, hk) - f(g,h) = 0$ for all $g, h, k \in G$.  A \emph{circular ordering} of $G$ is an inhomogeneous 2-cocycle that also satisfies $f(g, h) \in \{0,1\}$ for all $g, h \in G$ and $f(g, g^{-1}) = 1$ for all $ g \in G \setminus \{ id \}$.

Given a group $G$ with a circular ordering $f$, we can construct a left-ordered central extension $\widetilde{G}_f$ of $G$ as follows. Take the set $\mathbb{Z} \times G$, and equip it with a multiplication according to the rule $(n, g) (m,h) = (n+m+f(g,h), gh)$.  One can check that $P_f = \{ (n,g) \mid n \geq 0 \} \setminus \{ (0, id)\}$ defines a positive cone in $\widetilde{G}_f$, so that $\widetilde{G}_f$ is left ordered.  There is a central extension 
\[ 0 \lra \mathbb{Z} \stackrel{\iota}{\lra} \widetilde{G}_f \overset{q}\lra G \lra 1 
\]
where $\iota(n) = (n, id)$ and $q(n,g) = g$. The element $(1,id)$ is positive cofinal and central, and the floor of $(n,g) \in \widetilde G_f$ is given by $[(n,g)] = n$. 

This construction is nothing more than a special instance of the usual correspondence between elements of $H^2(G; \mathbb{Z})$ and equivalence classes of central extensions $0 \lra \mathbb{Z} \stackrel{\iota_1}{\lra} H \overset{q}\lra G \lra 1 $. Recall that two central extensions
\[
0 \lra \mathbb{Z} \overset{\iota_1}\lra H_1 \overset{\rho_1}\lra G \lra 1 \eand 0 \lra \mathbb{Z} \overset{\iota_2}\lra H_2 \overset{\rho_2}\lra G \lra 1
\]
are equivalent whenever there exists a homomorphism $\phi:H_1 \to H_2$ such that $\phi\iota_1 = \iota_2$ and $\rho_2\phi = \rho_1$, which in fact makes $\phi$ into an isomorphism.

With these lifts in hand, we can also define the rotation number of an element $g \in G$ relative to a circular ordering $f$ of $G$.  It is 
\[ \mathrm{rot}_g(f) = \tau_{\tilde{g}}(P_f) \mod \mathbb{Z}, 
\]
where $\tilde{g} \in \widetilde{G}_f$ is any element satisfying $q(\tilde{g}) = g$.  It is straightforward to check that this definition is independent of our choice of lift $\tilde{g}$, and that for all $k \in \ZZ$, $\mathrm{rot}_{g^k}(f) = k(\mathrm{rot}_{g}(f))$.

Since a circular ordering on a group $G$ takes values in $\{0,1\}$, it is bounded. Thus it defines an element $[f]$ in bounded cohomology $H^2_b(G;\ZZ)$. We say two circular orderings $f_1$, $f_2$ on $G$ are \emph{semiconjugate} if $[f_1] = [f_2] \in H^2_b(G;\ZZ)$. 

\begin{remark}
Given a circular ordering $f$ on a countable group $G$, one can define a dynamic realisation $\rho_f:G \to S^1$. There is a dynamical notion of two actions of $G$ on $S^1$ being semiconjugate. The two uses of the word semiconjugacy are compatible. More precisely, $f_1$ and $f_2$ are semiconjugate circular orderings on a group $G$ if and only if $\rho_{f_1}$ and $\rho_{f_2}$ are semiconjugate (see \cite{CG} for an exposition of this well-known fact).
\end{remark}

The next result gives a characterisation of semiconjugate circular orderings on a group $G$ in terms of translation numbers of their left-ordered central extension. The result is well-known from a dynamical perspective (see e.g. \cite{BFH} for an expository discussion in terms of bounded Euler classes and translation number, or \cite{Gh}).

\begin{proposition}\label{prop:translation-number-semi-conjugate}
Let $f_1$ and $f_2$ be circular orderings on a group $G$. Let $P_{f_1}$ and $P_{f_2}$ be the positive cones of the left-ordered central extensions $\widetilde G_{f_1}$ and $\widetilde G_{f_2}$. Then $[f_1] = [f_2] \in H^2_b(G;\ZZ)$ if and only if there exists an equivalence of central extensions $\theta:\widetilde G_{f_1} \to \widetilde G_{f_2}$ satisfying $\tau_{\tilde g}(P_{f_1}) = \tau_{\theta(\tilde g)}(P_{f_2})$ for all $\tilde g \in \widetilde G_{f_1}$.
\end{proposition}
\begin{proof}  Throughout this proof we will write $P_i$ in place of $P_{f_i}$ and $\widetilde G_i$ in place of $\widetilde G_{f_i}$ in order to simplify notation.

Suppose $f_1(g,h) - f_2(g,h) = d(g) + d(h) - d(gh)$ for some bounded function $d:G \to \ZZ$ (so $[f_1] = [f_2] \in H^2_b(G;\ZZ)$). Then $\theta:\widetilde G_1 \to \widetilde G_2$ given by $\theta(n,g) = (n + d(g),g)$ is an equivalence of central extensions. Let $(n,g) \in \widetilde G_1$. For $k \in \ZZ_{>0}$, let $(n,g)^k = (m_k,g^k)$. Then
\begin{align*}
\tau_{\theta(n,g)}(P_2) - \tau_{(n,g)}(P_1) &= \lim_{k \to \infty}\frac{1}{k}([\theta(n,g)^k]_{P_2} - [(n,g)^k]_{P_1}) \\
&= \lim_{k \to \infty}\frac{1}{k}([(m_k + d(g^k), g^k)]_{P_2} - [(m_k,g^k)]_{P_1}) \\
&= \lim_{k \to \infty}\frac{1}{k}(m_k + d(g^k) - m_k) \\
&= \lim_{k \to \infty}\frac{d(g^k)}{k}\\
&= 0
\end{align*}
since $d$ is bounded. Therefore $\tau_{\theta(n,g)}(P_2) = \tau_{(n,g)}(P_1)$.

Conversely, suppose $\theta:\widetilde G_1 \to \widetilde G_2$ is an equivalence of central extensions satisfying $\tau_{(n,g)}(P_1) = \tau_{\theta(g,n)}(P_2)$ for all $(n,g) \in \widetilde G_1$. Then $\theta$ takes the form $\theta(n,g) = (n + d(g),g)$ for a function $d:G \to \ZZ$ satisfying $f_1(g,h) - f_2(g,h) = d(g) + d(h) - d(gh)$ for all $g,h \in G$. We will show $d$ is bounded, and in particular, that $-1 \leq d(g) \leq 1$ for all $g \in G$.

Let $(n,g) \in \widetilde G_1$ and $k \in \ZZ_{>0}$. Then
\[
(n,g)^k = \left(kn + \sum_{i=1}^{k-1} f_1(g^i,g),g^k\right)\]
and
\[\theta((n,g))^k = \left(kn + kd(g) +  \sum_{i=1}^{k-1} f_2(g^i,g),g^k\right).
\]
The assumption $\tau_{(n,g)}(P_1) = \tau_{\theta(n,g)}(P_2)$ implies
\begin{align*}
0 &= \lim_{k\to \infty} \frac{1}{k} \left(kn + \sum_{i=1}^{k-1} f_1(g^i,g)\right) - \lim_{k\to \infty} \frac{1}{k} \left(kn + kd(g) +  \sum_{i=1}^{k-1} f_2(g^i,g)\right) \\
&= \lim_{k \to \infty} \frac{1}{k} \left(-kd(g) + \sum_{i=1}^{k-1}(f_1(g^i,g) - f_2(g^i,g))\right) \\
&= -d(g) + \lim_{k \to \infty} \frac{1}{k}\sum_{i=1}^{k-1}(f_1(g^i,g) - f_2(g^i,g)).
\end{align*}
Since $f_1(g^i,g),f_2(g^i,g) \in \{0,1\}$, we have
\[
-(k-1) \leq \sum_{i=1}^{k-1}(f_1(g^i,g) - f_2(g^i,g)) \leq k-1
\]
so 
\[
-1 \leq \lim_{k\to\infty} \frac{1}{k} \sum_{i=1}^{k-1}(f_1(g^i,g) - f_2(g^i,g)) \leq 1.
\]
Therefore $-1 \leq d(g) \leq 1$ for all $g \in G$, as desired.
\end{proof}

\subsection{Quotients of the braid group, circular orderings and lifts}
For $n \geq 2$ set $\delta_n = \sigma_1 \sigma_2 \cdots \sigma_{n-1}$.  Recall that the Garside element $\Delta_n$, whose square generates the centre of $B_n$, can be expressed as  $\Delta_n = \delta_n \delta_{n-1} \cdots \delta_2$.  Moreover,  we have $\delta_n^n = \Delta_n^2$ \cite[Lemma 4.4]{DDRW}.  The following lemmas allow us to bound the translation number of $\Delta_{n-1} \in B_n$.

\begin{lemma}
\label{prop1}
Suppose that $n \geq 3$ is an integer.  Then the following hold for all $k \in \mathbb{Z}$:
\begin{enumerate}
\item $\Delta_{n-1}^{-2(k+1)}\Delta_{n}^{2k}=\Delta_{n-1}^{-2}(\Delta_n \delta_n \Delta_n^{-1} \delta_n)^k$, and
    \item $\Delta_{n-1}^{2(k+1)}\Delta_{n}^{2k} = \Delta_{n-1}^2(\delta_n^{-1} \Delta_n \delta_n^{n-1}\Delta_n)^k$.
    
\end{enumerate}
\end{lemma}
\begin{proof}
To show (1), we first compute
 \begin{align*}
     \Delta_{n-1}^{-2}\Delta_n^2 &=\Delta_n^2 \Delta_{n-1}^{-2} \\
     &=(\delta_n \cdots \delta_2)(\delta_n \cdots \delta_2)(\delta_2^{-1} \cdots \delta_{n-1}^{-1})(\delta_2^{-1} \cdots \delta_{n-1}^{-1})\\
&= (\delta_n \cdots \delta_2)\delta_n (\delta_2^{-1} \cdots \delta_{n-1}^{-1})\\
&= \Delta_n \delta_n \Delta_n^{-1} \delta_n.\\
\end{align*}
Now we observe that 
\begin{align*}
    \Delta_{n-1}^{-2(k+1)}\Delta_{n}^{2k}&= \Delta_{n-1}^{-2}(\Delta_{n-1}^{-2}\Delta_{n}^{2})^k\\ 
    &= \Delta_{n-1}^{-2}(\Delta_{n}^{2}\Delta_{n-1}^{-2})^k\\  
    &= \Delta_{n-1}^{-2}(\Delta_n \delta_n \Delta_n^{-1} \delta_n)^k.
    \end{align*}
 To show (2), we begin with
 \begin{align*}
     \Delta_{n-1}^{2}\Delta_n^2 &=\Delta_n^2 \Delta_{n-1}^{2} \\
     &=\delta_n^n(\delta_{n-1} \cdots \delta_2)(\delta_{n-1} \cdots \delta_2)\\
&= \delta_n^{n-1}\Delta_n (\delta_{n-1} \cdots \delta_2)\\
&= \delta_n^{-1} \Delta_n \delta_n^{n} (\delta_{n-1} \cdots \delta_2)\\
&= \delta_n^{-1} \Delta_n \delta_n^{n-1} \Delta_n.\\
\end{align*}
Now we write 
\begin{align*}
    \Delta_{n-1}^{2(k+1)}\Delta_{n}^{2k}&= \Delta_{n-1}^{2}(\Delta_{n-1}^{2}\Delta_{n}^{2})^k\\ 
    &= \Delta_{n-1}^{2}(\Delta_{n}^{2}\Delta_{n-1}^{2})^k\\  
    &= \Delta_{n-1}^{2}(\delta_n^{-1} \Delta_n \delta_n^{n-1} \Delta_n)^k.
    \end{align*}    
\end{proof}

\begin{proposition}\label{prop2}
    If $(G, <)$ is a countable left-orderable group and $x, y \in G$ are positive and $<$-cofinal, then the product $xy$ is positive and $<$-cofinal.
\end{proposition}
\begin{proof} 


Suppose that $P$ is the positive cone of $<$, and note that $g\in G$ is $<$-cofinal if and only if $\rho_P(g)$ is without fixed points.  Therefore if $x, y \in G$ are positive and $<$-cofinal then $\rho_P(x)(a)>a$ and $\rho_P(y)(a)>a$ for any $a\in \mathbb{R}$. Thus $$\rho_P(xy)(a)=\rho_P(x)(\rho_P(y)(a))>\rho_P(y)(a)>a$$ for all $a\in \mathbb{R}$. Therefore $xy$ is $<$-cofinal. 
\end{proof}
\begin{proposition}\label{prop:Cext-Gn-Bn}
    If $<$ is a left ordering of $B_n$, $n \geq 3$ such that $\Delta_n^2>1$, then there exist positive integers $k, 
    \ell$ such that $\Delta_{n-1}^{2(k+1)}<\Delta_n^{2k}$ and $\Delta_n^{-2 \ell} < \Delta_{n-1}^{2(\ell+1)}$.
\end{proposition}
\begin{proof}
By Lemma \ref{prop1}(1), we have that  $$\Delta_{n-1}^{-2(m+1)}\Delta_{n}^{2m}=\Delta_{n-1}^{-2}(\Delta_n \delta_n \Delta_n^{-1} \delta_n)^m$$
    for all integers $m$. So to show that there exists a positive $k$ such that $\Delta_{n-1}^{-2(k+1)}\Delta_{n}^{2k}>id$, it suffices to show that the product  $\Delta_n \delta_n \Delta_n^{-1} \delta_n$ is positive and cofinal, for then there exists an integer $k>0$ such that $\Delta_{n-1}^{-2}(\Delta_n \delta_n \Delta_n^{-1} \delta_n)^k$ is positive. 
    
Note that $(\Delta_n \delta_n \Delta_n^{-1})^n=\Delta_n^2$ and $\delta_n^n=\Delta_n^2$. Since $\Delta_n^2$ is positive and confinal, both $\Delta_n \delta_n \Delta_n^{-1}$ and $\delta_n$ are positive and cofinal. Therefore, their product $(\Delta_n \delta_n \Delta_n^{-1})\delta_n$ is positive and cofinal by Proposition  \ref{prop2}.  This proves the existence of the required $k$.

Now by Lemma \ref{prop1}(2), we have
$$\Delta_{n-1}^{2(m+1)}\Delta_{n}^{2m} = \Delta_{n-1}^2(\delta_n^{-1} \Delta_n \delta_n^{n-1}\Delta_n)^m$$
for all integers $m$.  As above, in order to show the existence of the required $\ell>0$ it suffices to prove that $\delta_n^{-1} \Delta_n \delta_n^{n-1}\Delta_n$ is positive and cofinal.  Recall that $n \geq 3$, so by writing $\delta_n^{-1} \Delta_n \delta_n^{n-1}\Delta_n = (\delta_n^{-1}\Delta_n\delta)\delta_n^{n-2}\Delta_n$, we can observe that $\delta_n^{-1} \Delta_n \delta_n^{n-1}\Delta_n$ is a product of the positive, cofinal elements $\delta_n^{-1}\Delta_n\delta$, $\delta_n^{n-2}$ and $\Delta_n$, so is itself positive and cofinal as well.
\end{proof}

\begin{corollary}\label{cor:bounded-translation-number}
Suppose that $P \in B_n$ is the positive cone of a left ordering with $\Delta_n^2 \in P$.  Then $-1 < \tau_{\Delta_{n-1}}(P)< 1$.
\end{corollary}
\begin{proof}
By Proposition \ref{prop:Cext-Gn-Bn} we can choose $k>0$ such that $\Delta_{n-1}^{2(k+1)} <_P \Delta_n^{2k}$.  But then $\Delta_{n-1}^{2m(k+1)} <_P \Delta_n^{2mk}$ for all $m>0$, and so $[ \Delta_{n-1}^{2m(k+1)} ]_P < mk$.  Thus 
\[ \tau_{\Delta_{n-1}^2}(P) = \lim_{m \to \infty}\frac{[ \Delta_{n-1}^{2m(k+1)} ]_P}{m(k+1)} \leq \frac{k}{k+1} <1.
\]
Similarly by using Proposition \ref{prop:Cext-Gn-Bn} to choose $\ell>0$ such that $\Delta_n^{-2 \ell} < \Delta_{n-1}^{2(\ell+1)}$ we can show that $\tau_{\Delta_{n-1}^2}(P) > -1$.
\end{proof}




\begin{proposition} \label{prop:ext-is-bn}
Set $G_n = B_n/\langle \Delta_n^2 \rangle$ and let $f$ be an arbitrary circular ordering of $G_n$.  Then the central extension $\widetilde{(G_n)}_f$ is isomorphic to $B_n$.
\end{proposition}
\begin{proof}
The braid group $B_n$ has presentation 
\[ B_n = \left< \sigma_1 , \ldots, \sigma_{n-1}  \hspace{1em}
\begin{array}{|c}
\sigma_i \sigma_j = \sigma_j \sigma_i \mbox{ if $|i-j|>1$}\\
\sigma_i \sigma_j \sigma_i = \sigma_j \sigma_i \sigma_j \mbox{ if $|i-j|=1$} \end{array} 
 \right>,
\]
and upon setting $a = \delta_n = \sigma_1\dots \sigma_{n-1}$ and $b = \delta_n \sigma_1$, noting that $\delta_n \sigma_i = \sigma_{i+1} \delta_n$ for $1 \leq i \leq n-2$, we arrive at the presentation 
\[ B_n = \langle a, b \mid a^n = b^{n-1}, ba^{i-1}ba^{-i} = a^{i}ba^{-i-1}b \mbox{ for } 2 \leq i \leq n/2 \rangle.
\]
Therefore our group $G_n$ has the presentation 
\[ G_n = \langle a, b \mid a^n = b^{n-1} = id,  ba^{i-1}ba^{-i} = a^{i}ba^{-i-1}b \mbox{ for } 2 \leq i \leq n/2 \rangle.
\]
Now let $f$ be an arbitrary circular ordering of $G_n$, and recall that $\widetilde{(G_n)}_f$ is constructed as $\mathbb{Z} \times G_n$ with multiplication $(n, g)(m,h) = (n+m+f(g,h), gh)$, and that there is a short exact sequence
\[ 0 \longrightarrow \mathbb{Z} \stackrel{i}{\longrightarrow} \widetilde{(G_n)}_f \stackrel{q}{\longrightarrow} G_n \longrightarrow 1.
\]
We will denote the generator of the kernel of $q$ by $t$, so that $i(t) = (1, id)$, and recall that the quotient map acts by $q(n,g) = g$.

Recall from \cite[Definition 2.9]{CG2} that one may choose a \emph{minimal generator} of any circularly ordered cyclic group.  In this setting, this means we may choose $i \in \{1, \ldots, n-2\}$ and $j \in \{1, \ldots, n-3\}$ such that $a^i$ and $b^j$ are generators of $\langle a \rangle$ and $\langle b \rangle$ respectively, satisfying $f(a^i, a^{ki})=0$ for all $k = 1, \ldots, n-1$ and $f(b^j, b^{mj})=0$ for $m = 1, \ldots, n-2$.

A presentation for $\widetilde{(G_n)}_f$ is given by
\[ \langle x, y, t \mid \mbox{$t$ central, } x^n = t^u, y^{n-1}=t^v, yx^{i-1}yx^{-i} = x^{i}yx^{-i-1}yt^{k_i} \mbox{ for } 2 \leq i \leq n/2 \rangle
\]
where $u, v, k_i \in \mathbb{Z}$ are integers that are chosen so that $(0, a^i)^n = (u, id)$ and $(0, b^j)^n = (v, id)$ and 
\[ (0,b^j)(0,a^i)^{i-1}(0,b^j)(0,a^i)^{-i} = (0,a^i)^{i}(0,b^j)(0,a^i)^{-i-1}(0,b^j)(k_i, id)
\]
for  $2 \leq i \leq n/2$ \cite[Proposition 2.5]{HEO}.  First we observe that by our choice of $i, j$, we can compute the $n$-fold product: 
\[ (0, a^i) \cdots (0, a^i) = \left(\sum_{k=1}^{n-1}f(a^i, a^{ki}), a^{ni}\right) = (1, id).
\]
Here, the equality $\sum_{k=1}^{n-1}f(a^i, a^{ki}) =1$ follows from the fact that our choice of $i$ yields $f(a^i, a^{ki}) =
0$ for $1 \leq k \leq n-2$ since $a^i$ is a minimal generator of $\mathbb{Z} / n \mathbb{Z}$, and $f(a^i, a^{(n-1)i}) =1$.  Thus $u =1$.  We similarly compute that $v = 1$.  So eliminating the variable $t$ from our presentation for $\widetilde{(G_n)}_f$, we arrive at
\[ \langle x, y \mid  x^n = y^{n-1}, yx^{i-1}yx^{-i} = x^{i}yx^{-i-1}yx^{nk_i} \mbox{ for } 2 \leq i \leq n/2 \rangle.
\]
We conclude the proof by showing that the group defined by this presentation is not left orderable unless $k_i = 0$ for all $i$.

Given a left ordering of $\widetilde{(G_n)}_f$, we can assume that $x > id$ and therefore $y>id$ as well.  Moreover, both of these generators are positive and cofinal since the group $\widetilde{(G_n)}_f$ is generated by roots of the central element $x^n = y^{n-1}$, and therefore by Proposition \ref{prop2} every element in the semigroup $sg(x,y)$ is positive and cofinal as well.  In particular, for all $w \in sg(x,y)$ and for all $g \in \widetilde{(G_n)}_f$, we have $gwg^{-1} > id$ (though perhaps this conjugate is no longer cofinal).  We use these facts below.

 First suppose that $k_i >0$ for some $i$ with $2 \leq i \leq n/2$.  Then considering the relator $yx^{i-1}yx^{-i} = x^{i}yx^{-i-1}yx^{nk_i}$, we right-multiply by $y^{-1}x^{-i+1}y^{-1}$ to arrive at
\[ yx^{i-1}yx^{-i}y^{-1}x^{-i+1}y^{-1} = x^{i}yx^{-2i}y^{-1}x^{nk_i},
\]
here we have used that $x^n$ is central.  Now the left hand side is a conjugate of the negative element $x^{-i}$ and is thus negative.  On the other hand, using centrality of $x^n$ the right hand side becomes $x^i yx^{nk_i - 2i}y^{-1}$.  Recalling that $i \leq n/2$, the quantity $x^{nk_i - 2i}$ is either positive or the identity whenever $k_i >0$.  But then $x^i yx^{nk_i - 2i}y^{-1}$ is positive, a contradiction.

Next suppose that $k_i <0$ some $i$ with $2 \leq i \leq n/2$.  Again, we begin with $yx^{i-1}yx^{-i} = x^{i}yx^{-i-1}yx^{nk_i}$ and rearrange the expression to arrive at
\[x^{-nk_i-i}y^{-1}x^{1-i}yx^{i-1}yx^{-i} = yx^{-i-1}y.
\]
Observe that the left hand side can be written as 
\[x^{-nk_i-2i}(x^iy^{-1}x^{1-i}yx^{i-1}yx^{-i})
\]
and since $n >2$, $k_i < 0$ and $i\leq n/2$, $-nk_i-2i \geq 0$. Therefore this is a product of conjugates of cofinal, positive elements, and thus, every conjugate of the left hand side is positive.  Consequently, every conjugate of the right hand side, which is $yx^{-i-1}y$, must also be positive.

This implies that $x^{-i-1}y^2, y^{-2}x^{-i-1}y^4, \dots, y^{2(n-2)}x^{-i-1}y^{2(n-1)}$ are all positive elements.  Therefore their product
\[x^{-i-1}y^2 \cdot y^{-2}x^{-i-1}y^4  \dots y^{2(n-2)}x^{-i-1}y^{2(n-1)} = x^{(n-1)(-i-1)}y^{2(n-1)} = x^{(n-1)(-i-1) + 2n}
\]
is also positive.  However, $(n-1)(-i-1) + 2n = (2-i)n+(i-n)+1$, and $2-i \leq 0$ and $i -n < n/2-n <1$, so overall this quantity is negative.  But then $x^{(n-1)(-i-1) + 2n} < id$ is a contradiction, completing the proof.
%
%
\end{proof}

\subsection{Semiconjugacy of mapping class group actions on the circle}

We wish to apply the rigidity result of Theorem \ref{translation_number_same} to mapping class groups of marked spheres and once-marked surfaces of genus one.

 Let $\Sigma_{0,n}^1$ be a disk with $n$ marked points. Recall that there is an isomorphism $\varphi:B_n \to \Mod(\Sigma_{0,n}^1)$, such that if $d$ is a curve isotopic to $\partial\Sigma_{0,n}^1$, then $\varphi(\Delta_n^2) = T_d$. Here $T_d$ denotes the Dehn twist about $d$.

There is an inclusion $\epsilon:\Sigma_{0,n}^1 \hookrightarrow \Sigma_{0,n+1}$, where the complement of $\epsilon(\Sigma_{0,n}^1)$ is a once-marked open disk. Let $*$ denote the marked point in $\Sigma_{0,n+1}$ that is in the complement of $\epsilon(\Sigma_{0,n}^1)$. Let $G_n \subset \Mod(\Sigma_{0,n+1})$ be the stabiliser of $*$.

By extending homeomorphisms of $\epsilon(\Sigma_{0,n}^1)$ to $\Sigma_{0,n+1}$ by the identity, we get the so-called capping homomorphism $\mathcal{C}: \Mod(\Sigma_{0,n}^1) \to G_n$. The capping homomorphism gives rise to the central extension
\[
0 \longrightarrow \ZZ \overset{\iota}{\longrightarrow} \Mod(\Sigma_{0,n}^1) \overset{\mathcal{C}}{\longrightarrow} G_n \longrightarrow 1
\]
where $\iota(1) = T_d$ \cite[Section 3.6.2]{FM}. 

Since there exists a left ordering $<$ on $\Mod(\Sigma_{0,n}^1)$ so that $T_d$ is $<$-cofinal, $G_n$ is circularly orderable (in fact, as observed earlier, $T_d$ is $<$-cofinal for {\it every} left ordering $<$ of $\Mod(\Sigma_{0,n}^1)$).

For the next theorem, we first recall a well-known fact about circular orderings on a group $G$---namely that they come in pairs. Let $f:G^2 \to \ZZ$ be a circular ordering on $G$. Define the \emph{opposite ordering} $f^{\text{op}}:G^2 \to \ZZ$ by 
\[
f^{\text{op}}(g,h) = \begin{cases}
1 - f(g,h) &\text{if } g \neq \id, h \neq \id, \text{ and } gh \neq \id \\
f(g,h) &\text{otherwise}.
\end{cases}
\]
It can be checked that $f^{\text{op}}$ is indeed a circular ordering on $G$. Furthermore, if we let $d:G \to \ZZ$ be the function $d(\id) = 0$ and $d(g) = 1$ if $g \neq \id$, we have
\[
f(g,h) + f^{\text{op}}(g,h) = d(g) + d(h) - d(gh)
\]
for all $g,h \in G$. Since $d$ is bounded, we have $[f] = -[f^{\text{op}}] \in H_b^2(G;\ZZ)$, and therefore in $H^2(G;\ZZ)$.  Intuitively, we can think of $f^{\text{op}}$ as being obtained from $f$ by reversing the orientation of the circle. The next theorem is a restatement, using the notation developed in this section, of Theorem \ref{thm:intro_rigidity}.

\begin{theorem}
\label{main_rigidity}
Let $\alpha$ be a simple closed curve on $\Sigma_{0,n+1}$ that surrounds $n-1$ marked points, none of which are $*$. Let $f_1$ and $f_2$ be circular orderings on $G_n$ satisfying $ \mathrm{rot}_{T_\alpha}(f_1) = \mathrm{rot}_{T_\alpha}(f_2) = 0. $ Then $[f_1] = \pm[f_2] \in H^2_b(G_n;\ZZ)$.
\end{theorem}
\begin{proof}

Consider the central extensions
\[
0 \lra \ZZ \overset{\iota_1}{\lra} A_1 \overset{\rho_1}{\lra} G_n \lra 1 \eand 0 \lra \ZZ \overset{\iota_2}{\lra} A_2 \overset{\rho_2}{\lra} G_n \lra 1
\]
corresponding to $[f_1],[f_2] \in H^2(G;\ZZ)$ respectively. Let $j \in \{1,2\}$. Recall from the proof of Proposition \ref{prop:Cext-Gn-Bn} $G_n$ and $A_j$ admit group presentations
\begin{align*}
G_n &= \langle a, b \mid a^n = b^{n-1} = id,  ba^{i-1}ba^{-i} = a^{i}ba^{-i-1}b \mbox{ for } 2 \leq i \leq n/2 \rangle \\
A_j &= \langle x_j, y_j \mid  x_j^n = y_j^{n-1}, y_jx_j^{i-1}y_jx_j^{-i} = x_j^{i}y_jx_j^{-i-1}y_j \mbox{ for } 2 \leq i \leq n/2 \rangle
\end{align*}
respectively, with $\rho_j(x_j) = a$, $\rho_j(y_j) = b$, and $\iota_j(1) = x_j^n$. Let $\theta_j:B_n \to A_j$ be the isomorphism given by $\theta_j(\sigma_1\cdots\sigma_{n-1}) = x_j$ and $\theta_j(\sigma_1\cdots\sigma_{n-1}\sigma_1) = y_j$ as per Proposition \ref{prop:ext-is-bn}. Then $\rho_1\theta_1 = \rho_2\theta_2:B_n \to G_n$ is the capping homomorphism. By replacing $f_j$ by $f_j^{\text{op}}$ if necessary, we may assume $\theta_j(\Delta_n^2) = x_j^n = \iota_j(1)$. With this assumption, it suffices to show $[f_1] = [f_2] \in H^2_b(G_n;\ZZ)$. 

Let $Q_j$ be the positive cone of the left ordering on $A_j$ coming from the lifting construction of Subsection \ref{subsec:lifting_basics}, applied to the circular ordering $f_j$ on $G_n$. Then $\iota_j(1) = x_j^n \in Q_j$ is the corresponding positive cofinal central element. Consider the positive cone $P_j = \{g \in B_j: \theta_j(g) \in Q_j$. Then $\Delta_n^2 \in P_j$. Note that since $\theta_j(\Delta_n^2) = \iota_j(1)$, $\tau_g(P_j) = \tau_{\theta_j(g)}(Q_j)$ for all $g \in B_n$.

Consider now the element $\Delta_{n-1}^2 \in B_n$. As an element of $\Mod(\Sigma_{0,n}^1)$, $\Delta_{n-1}^2$ is a Dehn twist about a simple closed curve surrounding $n-1$ marked points. Since $p_j\theta_j$ is the capping homomorphism, $p_j\theta_j(\Delta_{n-1}^2) = T_{\beta}$ for some simple closed curve $\beta$ surrounding $n-1$ marked points in $\Sigma_{0,n+1}$. Observe that none of the $n-1$ marked points are in the complement of $\epsilon(\Sigma_{0,n}^1)$, where $\epsilon:\Sigma_{0,n}^1 \hookrightarrow \Sigma_{0,n+1}$ is the inclusion. Therefore, by the change of coordinates principle \cite[Section 1.3]{FM}, $T_{\beta}$ and $T_{\alpha}$ are conjugate in $G_n$. Since rotation number is a conjugacy invariant, we have $ \mathrm{rot}_{T_\beta}(f_j) = \mathrm{rot}_{T_\alpha}(f_j) = 0.$ So we have $\tau_{\Delta_{n-1}^2}(P_j) = \tau_{\theta_j(\Delta_{n-1}^2)}(Q_j) \in \ZZ$, which implies $\tau_{\Delta_{n-1}^2}(P_j) = 0$ by Corollary \ref{cor:bounded-translation-number}. By Corollary \ref{cor:agreewithDehornoy} we can conclude $\tau_g(P_1) = \tau_{g}(P_2)$ for all $g \in B_n$.

Finally, from the definition of $\theta_j$, we have $\theta_2\theta_1^{-1}:A_1 \to A_2$ satisfies $\theta_2\theta_1^{-1}\iota_1 = \iota_2$ and $\rho_2\theta_2\theta_1^{-1} = \rho_1$. Therefore $\theta_2\theta_1^{-1}$ is an isomorphism realising the equivalence of central extensions of $G_n$. For an arbitrary $x \in A_1$ we have
\[
\tau_x(Q_1) = \tau_{\theta_1^{-1}(x)}(P_1) = \tau_{\theta_1^{-1}(x)}(P_2) = \tau_{\theta_2\theta_1^{-1}(x)}(Q_2).
\]
By Proposition \ref{prop:translation-number-semi-conjugate} we may conclude $[f_1] = [f_2] \in H^2_b(G_n;\ZZ)$.
\end{proof}

\subsection{Genus 1 and the (projective) special linear group}

In the case $n = 3$, $G_3$ is isomorphic to the modular group $\operatorname{PSL}_2(\ZZ)$. Identify $\operatorname{PSL}_2(\ZZ)$ as the set of fractional linear transformations of the complex upper-half plane $z \mapsto \frac{qz + r}{sz + t}$, with $qt - rs = 1$ and $q,r,s,t \in \ZZ$. An explicit isomorphism $\phi:G_3 \rightarrow   \operatorname{PSL}_2(\ZZ)$ is given by
\[
\phi(a) = \left(z \mapsto \frac{-1}{z + 1}\right) \quad \text{and} \quad \phi(b) = \left(z \mapsto \frac{-1}{z}\right).
\]
The image of $\Delta_2^2 = \sigma_1^2$ in $G_3$ is given by $(a^{-1}b)^2$, and 
\[
\phi((a^{-1}b)^2) = \left(z \mapsto z-2\right).
\]
Let $F$ be the element $\left(z \mapsto z-2\right)$ in $\operatorname{PSL}_2(\ZZ)$. We get the following corollary of Theorem \ref{main_rigidity}.

\begin{corollary}\label{cor:PSL}
Let $f_1$ and $f_2$ be circular orderings on $\operatorname{PSL}_2(\ZZ)$ with the property that $\operatorname{rot}_F(f_1) = \operatorname{rot}_F(f_2) = 0$. Then $[f_1] = \pm[f_2] \in H^2_b(\operatorname{PSL}_2(\ZZ);\ZZ)$. \hfill \qedsymbol
\end{corollary}

There is a surjective homomorphism $\varphi:B_3 \to \operatorname{SL}_2(\ZZ)$ with $\ker(\varphi) = \langle \Delta_3^4\rangle$ given by $\varphi(\sigma_1) = \left[\begin{smallmatrix} 1 & 1 \\ 
0 & 1\end{smallmatrix}\right]$ and $\varphi(\sigma_1) = \left[\begin{smallmatrix} 1 & 0 \\ 
-1 & 1\end{smallmatrix}\right]$.

Exploiting the fact that $B_3$ is a central extension of $\operatorname{SL}_2(\ZZ)$, we can prove a theorem for $\operatorname{SL}_2(\ZZ)$ analagous to Theorem \ref{main_rigidity}.

\begin{theorem}\label{thm:SL2}
Let $A = \left[\begin{smallmatrix} 1 & 1 \\ 0 & 1 \end{smallmatrix}\right]$ and let $f_1$ and $f_2$ be circular orderings on $\operatorname{SL}_2(\ZZ)$ so that $\operatorname{rot}_A(f_1) = \operatorname{rot}_A(f_2) = 0$. Then $[f_1] = \pm[f_2] \in H_b^2(\operatorname{SL}_2(\ZZ);\ZZ)$.
\end{theorem}

The proof follows the same strategy as that of Theorem \ref{main_rigidity}, with many of the details being identical. However, there are a few key differences which we highlight in the proof.
\begin{proof}
Let 
\[
0 \lra \ZZ \overset{\iota_1}{\lra} H_1 \overset{\rho_1}{\lra} \operatorname{SL}_2(\ZZ) \lra 1 \eand 0 \lra \ZZ \overset{\iota_2}{\lra} H_2 \overset{\rho_2}{\lra} \operatorname{SL}_2(\ZZ) \lra 1
\]
 be the left-ordered central extensions of $\operatorname{SL}_2(\ZZ)$ corresponding to $f_1$ and $f_2$ respectively. For the remainder of this proof, let $j \in \{1,2\}$. Any torsion-free central extension of $\operatorname{SL}_2(\ZZ)$ is isomorphic to $B_3$. By replacing $f_j$ by $f_j^{\text{op}}$ if necessary, we can choose isomorphisms $\theta_j:B_3 \to H_j$ so that $\theta_j(\Delta_3^4) = \iota_j(1)$ and $\rho_j\theta_j(\sigma_1^2) = \left[\begin{smallmatrix} 1 & 2 \\ 0 & 1 \end{smallmatrix}\right] = A^2$. Note that $\mathrm{rot}_{f_j}(A^2) = 2\mathrm{rot}_{f_j}(A) = 0$.

 Let $Q_j$ be the positive cone on $H_j$, and let $P_j = \{b \in B_3 \mid \theta_j(b) \in Q_j\}$. Note that $P_j$ is a positive cone on $B_3$ with $\Delta_3^2 \in P_j$. Let $\mathcal{T}_b(P_j)$ denote the translation number of $b \in B_3$ treating $\Delta_3^4$ as the cofinal central element, and $\tau_b(P_j)$ the translation number treating $\Delta_3^2$ as the cofinal central element. Then $2\mathcal{T}_b(P_j) = \tau_b(P_j)$ for all $b \in B_3$. Since $\operatorname{rot}_{f_j}(A^2) = 0$, $\mathcal{T}_{\sigma_1^2}(P_j) \in \ZZ$ so $\tau_{\sigma_1^2}(P_j) \in 2\ZZ$. Therefore by Corollary \ref{cor:bounded-translation-number}, $\tau_{\sigma_1^2}(P_j) = 0$. It follows from Corollary \ref{cor:agreewithDehornoy} that $\tau_b(P_1) = \tau_b(P_2)$ for all $b \in B_3$, and so $\mathcal{T}_b(P_1) = \mathcal{T}_b(P_2)$. Therefore for all $g \in H_1$, $\tau_g(Q_1) = \tau_{\theta_2\theta_1^{-1}(g)}(Q_2)$. The proof is completed by noting that $\theta_2\theta_1^{-1}:H_1 \to H_2$ is an equivalence of central extensions and applying Proposition \ref{prop:translation-number-semi-conjugate}.
\end{proof}

Let $\Sigma$ be a genus-1 surface, or a genus-1 surface with one marked point. Then $\Mod(\Sigma) \cong \operatorname{SL_2(\ZZ)}$ via an isomorphism mapping a positive Dehn twist $T_a$ about a non-separating simple closed curve to $\left[\begin{smallmatrix} 1 & 1 \\ 0 & 1 \end{smallmatrix}\right]$ \cite[\S2.2.4]{FM}. Since positive Dehn twists about non-separating simple closed curves are conjugate in $\Mod(\Sigma)$ \cite[\S3.3]{FM}, we have the following corollary.

\begin{corollary}\label{cor:genus-1}
Let $\Sigma = \Sigma_1$ or $\Sigma_{1,1}$, and let $a$ be a non-separating simple closed curve on $\Sigma$. Let $f_1$ and $f_2$ be circular orderings on $\Mod(\Sigma)$ so that $\operatorname{rot}_{T_a}(f_1) = \operatorname{rot}_{T_a}(f_2) = 0$. Then $[f_1] = \pm[f_2] \in H_b^2(\Mod(\Sigma);\ZZ)$. \hfill \qedsymbol
\end{corollary}

\begin{remark} There is a natural action of $G_n$ on $S^1$ via the so-called conical cover (see \cite{BW}). Consider the $n+1$-punctured sphere $\Sigma_{0,n+1}$ with a distinguished puncture $*$, so $G_n$ is isomorphic to the subgroup of $\Mod(\Sigma_{0,n+1})$ that fixes $*$. Assuming $n \geq 3$, we can fix a hyperbolic metric on $\Sigma_{0,n+1}$. Choose a parabolic element $\gamma \in \pi_1(\Sigma_{0,n+1})$ corresponding to the distinguished puncture $*$, and let $\widetilde\Sigma_{0,n+1}$ be the cover corresponding to the cyclic subgroup $\langle \gamma \rangle \subset \pi_1(\Sigma_{0,n+1})$. The covering space $\widetilde\Sigma_{0,n+1}$ is called the \emph{conical cover} at the puncture $*$. The conical cover is homeomorphic to an open punctured disk and has a hyperbolic metric obtained by pulling back the metric from $\Sigma_{0,n+1}$. The Gromov boundary of $\widetilde \Sigma_{0,n+1}$ is the disjoint union of a point (corresponding to the puncture $*$) and a circle. The boundary circle parametrises geodesic rays starting from $*$, and every ray out of $*$ has a unique preferred lift to $\widetilde \Sigma_{0,n+1}$. By lifting elements of $G_n$ to $\widetilde\Sigma_{0,n+1}$, we get an action of $G_n$ on the boundary circle. Consider now a simple closed curve $\alpha$ on $\Sigma_{0,n+1}$ surrounding $*$ and another puncture as in Theorem \ref{main_rigidity}. There is a ray $\delta$ from $*$ to the other puncture that is disjoint from $\alpha$. Therefore $T_\alpha$ fixes $\delta$, and the lift of $T_\alpha$ to $\widetilde \Sigma_{0,n+1}$ fixes the lift of $\delta$. This gives a fixed point of the action of $T_\alpha$ on $S^1$, and so the rotation number of $T_\alpha$ is $0$. By Theorem \ref{main_rigidity}, every action of $G_n$ on $S^1$ such that $T_\alpha$ acts with rotation number $0$ is semiconjugate (up to reversing the orientation of the circle) to the action coming from the conical cover. 

Similarly, $\mathrm{SL}_2(\ZZ)$ acts on the set of half rays in $\RR^2$, which is circularly ordered. The element $A = \left[\begin{smallmatrix} 1 & 1 \\ 0 & 1\end{smallmatrix}\right]$ fixes the half ray passing through the point $(1,0)$, and so the rotation number of $A$ corresponding to this action is $0$. Theorem \ref{thm:SL2} implies that any action of $\mathrm{SL}_2(\ZZ)$ on $S^1$ such that $A$ acts with rotation number 0 is semiconjugate, up to reversing the orientation of the circle, to the action on the set of half rays.

The modular group $\mathrm{PSL}_2(\ZZ)$ acts on the upper half complex plane $\HH^2$ by isometries and therefore on its boundary $\partial\HH^2 \simeq S^1$. The element $F = (z \mapsto z - 2)$ fixes $\infty$ and therefore $F$ acts with rotation number $0$. We can conclude from Corollary \ref{cor:PSL} that up to reversing the orientation of the circle, any action on $S^1$ with the property that $F$ acts with rotation number $0$ is semiconjugate to the action on $\partial \HH^2$ described above.
\end{remark}

\end{document}